\documentclass[11pt,leqno]{amsart}

\usepackage{amsmath,amsthm,amsfonts,amssymb,latexsym}

\newtheorem{theorem}{Theorem}[section]
\newtheorem{lemma}[theorem]{Lemma}
\newtheorem{prop}[theorem]{Proposition}
\newtheorem{cor}[theorem]{Corollary}
\theoremstyle{remark}
\newtheorem{remark}[theorem]{Remark}

\def\R{{\mathbb R}}

\def\O{{\mathcal O}}
\def\norm[#1][#2]{\| #1\|_{#2}}
\def\bignorm[#1][#2]{\bigg\| #1 \bigg\|_{#2}}
\def\ptl[#1][#2]{\frac{\partial #1}{\partial #2}}
\def\japanese[#1]{\langle #1 \rangle}

\def\eps{\varepsilon}

\def\la{\langle}
\def\ra{\rangle}
\def\les{\lesssim}
\def\ges{\gtrsim}
\def\calR{{\mathcal R}}
\def\Z{{\mathbb Z}}
\def\dist{{\rm dist}}
\def\supp{{\rm supp\;}}
\def\PV{{\rm P.V.}}
\def\calJ{{\mathcal J}}
\def\sign{{\rm sign}\;}
\def\calB{{\mathcal B}}
\def\calC{{\mathcal C}}
\def\half{\frac12}
\renewcommand{\tilde}{\widetilde}

\numberwithin{equation}{section}

\begin{document}

\title[Strichartz estimates for almost critical magnetic potentials]
{Strichartz and smoothing estimates for Schr\"odinger operators with
almost critical magnetic potentials in three and higher dimensions.}
\author{M. Burak Erdo\smash{\u{g}}an, Michael Goldberg, Wilhelm Schlag}
\thanks{The first author was partially supported by NSF grant DMS-0600101,
the second author by NSF grant DMS-0600925,
and the third author by NSF grant DMS-0617854.}
\address{Department of Mathematics \\
University of Illinois \\
Urbana, IL 61801, U.S.A.}
 \email{berdogan@math.uiuc.edu}

\address{Department of Mathematics\\ Johns Hopkins University\\
Baltimore, MD 21218, U.S.A.}
\email{mikeg@math.jhu.edu}

\address{ Department of Mathematics, University of Chicago,
5734 South University Avenue, Chicago, IL 60637, U.S.A.}
\email{schlag@math.uchicago.edu}

\begin{abstract}
In this paper we consider
Schr\"odinger operators \[H = -\Delta+ i(A\cdot\nabla + \nabla\cdot A) + V=-\Delta+L\]
in $\R^n$, $n\ge3$. Under almost optimal conditions on $A$ and $V$ both
in terms of decay and regularity we prove smoothing and Strichartz estimates,
as well as a limiting absorption principle. For large gradient perturbations the latter
is not an immediate corollary of the free case as $T(\lambda):=L(-\Delta-(\lambda^2+i0))^{-1}$
is not small in operator norm on weighted $L^2$ spaces as $\lambda\to\infty$.
We instead deduce the existence of inverses $(I + T(\lambda))^{-1}$ by showing
that the spectral radius of $T(\lambda)$ decreases to zero.  In particular, 
there is an integer $m$ such that $\limsup_{\lambda\to\infty}\|T(\lambda)^m\|<\frac12$.
This is based on an angular decomposition of the free resolvent for which we
establish the limiting absorption bound
\begin{equation}\label{eq:introeq}
\| D^{\alpha}\calR_{d,\delta}(\lambda^2)f \|_{B^*} \le
C_n\lambda^{-1+|\alpha|} \|f\|_B
\end{equation}
where $0\le|\alpha|\le 2$, $B$ is the Agmon-H\"ormander space, and $\calR_{d,\delta}(\lambda^2)$
is the free resolvent operator at energy $\lambda^2$ whose kernel is restricted in angle to a cone
of size $\delta$ and by $d$ away from the diagonal $x=y$. The main point is that $C_n$ only
depends on the dimension, but not on the various cut-offs. The proof of~\eqref{eq:introeq} avoids
 the Fourier transform and instead uses H\"ormander's variable coefficient
Plancherel theorem for oscillatory integrals.
\end{abstract}

\maketitle

\section{Introduction}

In this paper we prove Strichartz and smoothing bounds for the
magnetic Schr\"odinger operator on $L^2(\R^n)$
\begin{equation}
\label{eq:schr1} H = -\Delta+ i(A\cdot\nabla + \nabla\cdot A) + V = -\Delta + L
\end{equation}
under almost optimal assumptions on the large perturbations $A$ and
$V$. As usual we will assume that zero energy is neither an
eigenvalue nor a resonance. This means that the perturbed resolvent
$(H-z)^{-1}$ remains bounded on the weighted spaces $L^{2,1+}\to
L^{2,1-}$ as $z\to0$, $\Im z>0$. This condition is equivalent
(assuming sufficient decay on $A,V$)
 to the absence of nonzero solutions $f$ of $Hf=0$ with
$f\in  L^{2,\frac{n-4}{2}-}$.  When $n\ge 5$ any such solution
belongs to $L^2(\R^n)$ itself, so it suffices to check that zero
energy is not an eigenvalue.

\begin{theorem}
\label{thm:main} Let $A$ and $V$ be real-valued such that for all
$x\in\R^n$, $n\ge3$, and some fixed but arbitrary $\eps>0$ and all
sufficiently small $0<\eps'<\eps$,
  \begin{align}
  |A(x)|+ \la x\ra |V(x)| &\les \la x\ra^{-1-\eps}  \label{eq:ass1} \\
\la x\ra^{1+\eps'}A(x) &\in \dot W^{\half, 2n}(\R^n) \label{eq:ass2} \\
A &\in C^0(\R^n)\label{eq:ass3}
\end{align}
Furthermore, assume that zero energy is neither an eigenvalue nor a
resonance of $H$. Then, with $P_c$ being the projection onto the continuous
spectrum,
\begin{equation}\label{eq:strich} \|e^{itH} P_c f\|_{L_t^q(L_x^p)}
\les \|f\|_{L^2(\R^n)} \end{equation} provided
$\frac{2}{q}+\frac{n}{p}=\frac{n}{2}$ and $2\le p<\frac{2n}{n-2}$.
Moreover, the Kato smoothing estimates
\begin{equation}\label{eq:smooth}\begin{split}
\int_0^\infty \big\|\la x\ra^{-\sigma} | \nabla |^{\frac12}
e^{itH}P_c f\big\|_2^2\, dt &\le C\|f\|_2^2\\
\int_0^\infty \big\|\la x\ra^{-2\sigma}\la \nabla\ra^{\frac12}
e^{itH}P_c f\big\|_2^2\, dt &\le C\|f\|_2^2
\end{split}
\end{equation}
hold with $\sigma>\frac12$.
\end{theorem}

The secondary condition \eqref{eq:ass2} deals with the
regularity of $A$ but does not impose any extra decay beyond what is
assumed in~\eqref{eq:ass1}.  Note that $\la x\ra^{1+\eps'}A(x)$
must decay like $\la x\ra^{-(\eps-\eps')}$ which already belongs to $\dot
W^{\frac12,2n}(\R^n)$.  A stronger, but more easily verifiable hypothesis would
be to require $A$ to be Lipschitz continuous with
$|\nabla A(x)| \les \japanese[x]^{-2-\eps}$.  However stated,
this condition permits the commutation of $A$ with
$|\nabla|^{\frac12}$, which is essential to our factorization of $L$ into
pairs of pseudo-differential operators each having order $\frac12$.
 The continuity
assumption~\eqref{eq:ass3} is required for our treatment of large
energies. To relax it, one needs to carry out some of our large
energy analysis on spaces other than the $L^2$ based spaces $B,B^*$
which we use here. See~\cite{IS} for such work on Stein-Tomas type
spaces. 
%However, since $W^{\frac12+, 2n}(\R^n)$ and 
%imbeds into~$C^0(\R^n)$, we ignore this issue for now.

Since $L^1\to L^\infty$ dispersive bounds are currently
unknown for any $A \not= 0$, we cannot follow the usual interpolation method.
Instead, we adopt an argument introduced in~\cite{RS}, where the validity of
Strichartz inequalities is instead derived from Kato's theory of smooth
perturbations. This paper is related to our three-dimensional paper~\cite{EGS},
where a result similar to Theorem~\ref{thm:main} was proved but under much
stronger conditions on $A,V$, both in terms of decay as well as regularity.
In
\cite{Stef} and~\cite{GST} Strichartz and smoothing estimates were
obtained for {\em small} $A$ and $V$. For more background and many references
on magnetic operators see Erd\"os's survey~\cite{Erdos}.

The approach of this work is perturbative around the free case
despite the fact that we make no smallness assumption.
The main novel ingredient in
this paper is a limiting absorption estimate for large energies on almost
optimal weighted spaces.
To see the difficulty with large energies, recall that in~\cite{Agmon} and~\cite{IS} it is
proved that for $H$ as in~\eqref{eq:schr1} under suitable decay
conditions on $A$ and $V$ and with $\sigma>\frac12$,
\begin{equation}\label{eq:limap0} \sup_{\lambda\in[\delta,\delta^{-1}]} \|
\la x\ra^{-\sigma} \la \nabla\ra (H-(\lambda^2+i0))^{-1}
\la \nabla\ra \la x\ra^{-\sigma} \|_{2\to2} \le C(\delta)<\infty
\end{equation}
provided there are no imbedded eigenvalues in the continuous
spectrum (which is known due to recent work by Koch and
Tataru~\cite{KochT}).  It is well--known that this {\em limiting
absorption principle} is of fundamental importance for proving
dispersive estimates, at least for the case of large potentials.
However, one needs to consider all real $\lambda$ instead of
restricting to a compact interval in the positive halfline. To
extend~\eqref{eq:limap0} toward zero energies is similar to the case
$A=0$. This step requires the assumption on zero energy.

Note that~\eqref{eq:limap0} as stated cannot be extended to a semi-infinite
interval since it would fail even for the free resolvent.
Indeed, with $\sigma >\frac12$
\begin{equation}\label{eq:freeone}
\|\la x\ra^{-\sigma} \la \nabla\ra^{\alpha}
(H_0-(\lambda^2+i0))^{-1}\la \nabla\ra^{\alpha}
\la x\ra^{-\sigma}\|_{2\to2} \sim \lambda^{2\alpha - 1}
\end{equation}
for any $\alpha \in [0,1]$ and all $\lambda > 1$.  This shows that
no more than one derivative in total can be gained here while still preserving
a uniform upper bound.  Furthermore, in the borderline case $\alpha = \frac12$
there is no decay of the operator norm in the limit $\lambda \to \infty$. This is the
main difficulty we face when $A$ and $\lambda$ are large.

We will adopt the shorthand notation
\[ R_0(z) := (H_0 - z)^{-1} \]
for the resolvent of the Laplacian.  The resolvent of a general operator
$H$ will be indicated by $R_H(z)$, or else $R_L(z)$ in the case where
$H$ is specifically of the form $H_0 + L$.  Formally, the relationship between
$R_L$ and $R_0$ is captured in the identity
\[R_L(z) = (I + R_0(z)L)^{-1} R_0(z).\]

In this paper we extend~\eqref{eq:freeone}
to~$H = H_0 + L$ for the class of first-order perturbations described in
Theorem~\ref{thm:main}.  A unified statement of the mapping properties
of the resolvent of $H$ over the entire spectrum $\lambda > 0 $ is as follows.

\begin{theorem} \label{thm:limap}
Suppose $H$ is a magnetic Schr\"odinger operator whose potentials satisfy the
conditions $\japanese[x]^{1+\eps}(|A| + |V|) \in L^\infty(\R^n)$ with $A$
being continuous.  Then for $\sigma > \frac12$
and $\alpha \in [0, \frac12]$,
\begin{equation} \label{eq:limap1}
\sup_{\lambda > 1} \, \lambda^{1-2\alpha}
\| \la x\ra^{-\sigma} \la\nabla\ra^{\alpha} (H - (\lambda^2 + i0))^{-1}
   \la\nabla\ra^{\alpha} \la x\ra^{-\sigma}\|_{2\to2} \les 1.
\end{equation}
If one further assumes that $A$, $V$ satisfy the full conditions of
Theorem~\ref{thm:main}, and zero is not an eigenvalue or resonance of $H$,
then this bound can be extended to $\lambda \in [0,\infty]$
in either of two ways:
\begin{equation} \label{eq:limap2} \begin{split}
\sup_{\lambda \ge 0} \ \la \lambda \ra^{1-2\alpha}
\| \la x \ra^{-\sigma} |\nabla|^{\alpha} (H - (\lambda^2 + i0))^{-1}
   |\nabla|^{\alpha} \la x\ra^{-\sigma}\|_{2\to2} \les 1 \\
\sup_{\lambda \ge 0} \ \la \lambda \ra^{1-2\alpha}
\| \la x \ra^{-2\sigma}\la\nabla\ra^{\alpha} (H - (\lambda^2 + i0))^{-1}
   \la\nabla\ra^{\alpha} \la x\ra^{-2\sigma}\|_{2\to2} \les 1
\end{split}
\end{equation}
As a consequence, the spectrum of $H$ is purely absolutely continuous over the
entire interval $[0,\infty)$, and both the operators $\japanese[x]^{-\sigma}
|\nabla|^{\frac12}$ and $\japanese[x]^{-2\sigma}\japanese[\nabla]^{\frac12}$
are $H$-smooth on this interval.
\end{theorem}

\begin{remark} In fact, \eqref{eq:limap1} is valid for any
$\alpha \in [0,1]$, and with a somewhat wider class of potentials than is
described here.  This is made evident in the statement and proof of
Corollary~\ref{cor:highenZ}.
Furthermore, it is only necessary to verify the first assertion in
\eqref{eq:limap2}.  The second line follows immediately because the operator
$\japanese[x]^{-2\sigma}\japanese[\nabla]^{\frac12}|\nabla|^{-\frac12}
\japanese[x]^{\sigma}$ and its transpose are bounded on $L^2$,
see Lemma~\ref{lem:frac}.
\end{remark}
\begin{remark} A result of type \eqref{eq:limap1}, in the case $\alpha=0$,
 is proved in \cite{Rob}
using the method of Mourre commutators and micro-local analysis.  In
that work the potentials require only very slight polynomial decay,
however they are also assumed to be infinitely differentiable, with
the derivatives satisfying a symbol-like decay condition.
\end{remark}

This paper is organized as follows: In Section~\ref{sec:setup} we
present the reduction of the Strichartz estimates to the Kato
smoothing property~\cite{Kato}. More precisely, we are reduced to
proving that $Z_0:=\la x\ra^{-\sigma}|\nabla|^{\frac12}$ is
smoothing relative to $H$ for $\sigma>\frac12$ (it is a classical
result that $Z_0$ is smoothing relative to $H_0$). In
Section~\ref{sec:directed} we establish our main technical
ingredient, i.e., the limiting absorption principle for the
angularly truncated free resolvent kernel. It is essential here that
the bound does not deteriorate as the size of the truncation
decreases to zero. 

In Section~\ref{sec:high_limap} we use this bound
to prove a limiting absorption principle for the perturbed resolvent
via the ``power method'', i.e., by showing that $(LR_0)^m$ has small
norm for large energies and large~$m$. The idea is to write this
power as a sum of products involving conically restricted free
resolvents and to obtain a gain for both the ``directed''  (where
all the factors have almost aligned cones) and the ``undirected''
summands. In the former case this takes the form of a Volterra-type
gain, whereas in the latter one exploits a gain coming from angular
separation (for this one needs Schwartz potentials and general $A$
that are approximated by Schwartz functions; it is here that $A\in
C(\R^n)$ is needed). 

Finally, Section~\ref{sec:small} presents the
low energy case. Although this is similar to the case of $A=0$ in
that we use Fredholm's alternative and a Neumann series, it does
have some challenges of its own mainly in form of commutator
estimates. Finally, in the appendix we collect some tools from
harmonic analysis.

\section{The basic setup}
\label{sec:setup}

The Strichartz estimates stated in Theorem~\ref{thm:main} will be proved using
Proposition~\ref{prop:Katoth}
below, which was proved in \cite{RS}, see Theorem~4.1 in that paper. It is based on
Kato's notion of smoothing
operators, see~\cite{Kato}. We recall that for a self-adjoint operator $ H$,
an operator $\Gamma$ is called
$H$-smooth in Kato's sense if for any $f\in {\mathcal D}(H_0)$
\begin{equation}
\label{eq:Csmooth} \|\Gamma e^{it H} f\|_{L^2_t L^2_x}\le
C_{\Gamma}(H) \|f\|_{L^2_x}
\end{equation}
or equivalently, for any $f\in L^2_x$
\begin{equation}
\label{eq:CsmoothR} \sup_{\eps >0} \|\Gamma R_{H}(\lambda\pm i\eps)
f\|_{L^2_\lambda L^2_x} \le C_{\Gamma}(H)\|f\|_{L^2_x}.
\end{equation}
We shall call $C_{\Gamma}(H)$ the smoothing bound of $\Gamma$
relative to $H$. Let $\Omega\subset \R$ and let $P_\Omega$ be a
spectral projection of $H$ associated with a set $\Omega$. We say
that  $\Gamma$ is $H$-smooth on $\Omega$ if $\Gamma P_{\Omega}$ is
$H$-smooth. We denote the corresponding smoothing bound by
$C_\Gamma(H,\Omega)$. It is not difficult to show (see
e.g.~\cite{RS4}) that, equivalently,  $\Gamma$ is $H$-smooth on
$\Omega$ if
\begin{equation}
\label{eq:smon} \sup_{\beta >0} \|\chi_\Omega(\lambda)\Gamma
R_{H}(\lambda\pm i\beta ) f\|_{L^2_\lambda L^2_x} \le C_{\Gamma}(H,
\Omega)\|f\|_{L^2_x}.
\end{equation}

The estimate \eqref{eq:strich} of Theorem~\ref{thm:main} is obtained
by means of the following result. The remainder of the paper is
devoted to verifying the conditions needed in
Proposition~\ref{prop:Katoth}. Furthermore, this verification will
establish the smoothing estimate~\eqref{eq:smooth}.

\begin{prop}\label{prop:Katoth}
Let $H_0=-\Delta$ and $H= H_0 + L$ with $L=\sum_{j=1}^J Y_j^* Z_j$.
We assume that each $Y_j$ is $H_0$-smooth with a smoothing bound
$C_B(H_0)$ and that for some $\Omega\subset \R$ the operators $Z_j$
are $H$-smooth on $\Omega$ with the smoothing bound $C_A(H,\Omega)$.
Assume also that the unitary semigroup $e^{it H_0}$ satisfies the
estimate
\begin{equation}
\label{eq:strH0} \|e^{it H_0} \psi_0\|_{L^q_t L^r_x} \le C_{H_0}
\|\psi_0\|_{L^2_x}
\end{equation}
for some $q\in (2,\infty]$ and $r\in [1,\infty]$. Then the semigroup
$e^{itH}$ associated with $H = H_0 + L$, restricted to the spectral
set $\Omega$, also verifies the estimate \eqref{eq:strH0}, i.e.,
\begin{equation}
\label{eq:strHk} \|e^{it H}P_\Omega \psi_0\|_{L^q_t L^r_x} \le J
C_{H_0}C_B(H_0) C_A(H,\Omega) \|\psi_0\|_{L^2_x}
\end{equation}
\end{prop}

We refer the reader to \cite{RS} for the proof.

To apply this proposition we write,
with a decreasing weight $w(x)=\la x\ra^{-\tau}$ chosen from the range
 $\tau \in (\frac12, \frac12+\eps')$,
\begin{equation}\label{eq:Ldecomp}
 i(A\cdot\nabla + \nabla\cdot A) = \sum_{j=1}^2 Y_j^* Z_j,\quad V= Y_3^* Z_3
\end{equation}
where
\begin{equation}
\label{eq:YZdef}
\begin{aligned}
 & Y_1^* := i Aw^{-1} \cdot \nabla |\nabla|^{-\frac12}, \quad Z_1
  := |
 \nabla|^{\frac12} w \\
 & Y_2 := Z_1,\quad Z_2 := Y_1, \quad Y_3 = |V|^{\frac12}\sign V,\quad Z_3= |V|^{\frac12}
\end{aligned}
\end{equation}
Note that the cross-term produced by $Y_1^*Z_1$ is point-wise
multiplication by the purely imaginary function $i(Aw^{-1} \cdot
(\nabla w))$.  It is canceled by the corresponding cross-term in
$Y_2^*Z_2$.

We now reduce the smoothing properties of $Y_j$ and $Z_j$, $1 \le j \le 3$,
relative to $H_0$ and $H$, respectively, to the smoothing properties of
\begin{equation}\label{eq:Z0}
Z_0 := \la x\ra^{-\sigma}|\nabla|^{\frac12},
\end{equation}
where $\sigma$ is chosen so that $\frac12 < \sigma < \tau$.
It is standard that $Z_0$
is smoothing relative to $H_0$.
Theorem~\ref{thm:limap}, once proven, demonstrates that $Z_0$
is also smoothing relative to~$H$.
We first state a technical lemma which
explains the role of our regularity assumption~\eqref{eq:ass2}.

\begin{lemma}
  \label{lem:Amap} Let $A,\eps'$ be as in Theorem~\ref{thm:main}.
 Then the operator \[|\nabla|^{\half} A\la
  x\ra^{1+\eps'} |\nabla|^{-\half}
\]
is bounded
  on~$L^2$.
\end{lemma}
\begin{proof}
This is a straightforward application of the fractional Leibniz rule
in Lemma~\ref{lem:Leib}.
\[\begin{split}
\|\,|\nabla|^{\frac12} \tilde A  f\|_2 &\les \|\tilde A\|_\infty
\|\,|\nabla|^{\frac12} f\|_2 +
\|\,|\nabla|^{\frac12}\tilde A\|_{2n} \|f\|_{\frac{2n}{n-1}}\\
&\les ( \|\tilde A\|_\infty  + \|\,|\nabla|^{\frac12}\tilde A\|_{2n}
)\|\,|\nabla|^{\frac12} f\|_2 \les  \||\nabla|^{\frac12} f\|_2
\end{split}\]
which is equivalent to $|\nabla|^{\frac12} \tilde A
|\nabla|^{-\frac12}\::L^2\to L^2$.
\end{proof}

Returning to our discussion of the decomposition of~$L$,  observe
that
\begin{align*}
Z_1 &= \big(|\nabla|^{\frac12} w |\nabla|^{-\frac12} \la x\ra^{\sigma}\big)
\la x\ra^{-\sigma}|\nabla|^{\frac12} =: S_1  Z_0 \\
Z_2 &= i\big(\nabla|\nabla|^{-\frac12} Aw^{-1} |\nabla|^{-\frac12}
\la x\ra^{\sigma}\big) \la x\ra^{-\sigma}|\nabla|^{\frac12}=: S_2  Z_0 \\
Z_3 &= \big(|V|^{\frac12} |\nabla|^{-\frac12} \la x\ra^{\sigma}\big)
\la x\ra^{-\sigma} |\nabla|^{\frac12} =: S_3  Z_0
\end{align*}
with $S_1$ being $L^2$ bounded by Lemma~\ref{lem:comm} in the appendix.
Similarly, $S_3$ can be expanded as
\[
S_3 = i\big(|V|^{\frac12} w^{-1} |\nabla|^{-\frac12}\big) S_1
\]
and the operator in parentheses is bounded on $L^2$ by fractional integration.
For $S_2$, we need to invoke the local regularity of $A$:
\[
S_2=\nabla|\nabla|^{-\frac12} Aw^{-1}\la x\ra^{(1+\eps')-\tau}
|\nabla|^{-\frac12} \big(|\nabla|^{\frac12} \la x\ra^{\tau-(1+\eps')}
|\nabla|^{-\frac12} \la x\ra^{\sigma}\big),
\]
and the operator in parentheses is again $L^2$ bounded by Lemma~\ref{lem:comm},
whereas, by \eqref{eq:ass2} we can rewrite the remaining expression on the
right-hand side as
\[ \nabla|\nabla|^{-\frac12} A\la x\ra^{1+\eps'} |\nabla|^{-\frac12} =
\sum_{j=1}^n \partial_j|\nabla|^{-1} |\nabla|^{\frac12} A\la x \ra^{1+\eps'}
|\nabla|^{-\frac12}.
\]
The sum here is $L^2$ bounded; indeed, obviously the Riesz transforms
$\partial_j|\nabla|^{-1}$ are $L^2$ bounded and now apply
Lemma~\ref{lem:Amap}. In conclusion it will suffice to prove that
$Z_0$
is $H$-smooth.

Let us first consider intermediate energies $\lambda^2$, i.e.,
$\lambda\in[\lambda_0,\lambda_1]=\calJ_0$ with $\lambda_0$ small and
$\lambda_1$ large. Then it was shown in \cite{IS}, see
also~\cite{Agmon}, that the resolvent of $H$ satisfies the following
bound
\[
\sup_{\lambda\in\calJ_0}\|\la x\ra^{-\sigma}\la \nabla\ra
R_L(\lambda^2+i0) f\|_2 \le C(\lambda_0,\lambda_1)\, \|\la
x\ra^{\sigma} \la \nabla\ra^{-1} f\|_2
\]
(in fact, a stronger bound was proved in \cite{IS}). More precisely,
this bound follows provided there are no eigenvalues of $H$ in the
interval $\calJ_0$.
The latter property (absence of imbedded eigenvalues)
is shown in~\cite{KochT} to hold for the entire family of potentials
under consideration.  It is not difficult to replace the derivative
$\japanese[\nabla]$ with $|\nabla|$, since the operator
$|\nabla|\japanese[\nabla]^{-1}$ is bounded on a wide range of weighted
$L^2$ spaces, see Lemma~\ref{lem:frac}. Thus,
\begin{equation}\nonumber
\sup_{\lambda\in\calJ_0} \|Z_0 R_L(\lambda^2+i0) Z_0^*
\|_{2\to2}
 \le C(\lambda_0,\lambda_1)
\end{equation}
Finally, by Kato's smoothing theory, see \cite{RS4} Theorem~XIII.30,
we conclude that $Z_0$ is $H$-smooth on $\Omega=\calJ_0$ as desired.
In the following two sections we treat the case of large energies,
which takes up the most work.  The small energy case is then treated
in Section~\ref{sec:small}. Finally, in the appendix we collect some
bounds from harmonic analysis. Although they can all be found in the
literature in some form, the specific version required here appears
to be somewhat different.

\section{The Directed Resolvent Estimate}\label{sec:directed}

This section, which can be read independently of the other sections,
presents a limiting absorption estimate for the truncated free
resolvent kernel. The crucial point is that the constants in our
estimate do not depend on the truncation. Our main tool is
H\"ormander's variable coefficient Plancherel theorem from the
appendix.

The kernel of the free resolvent $R_0^+(\lambda^2)$ in $\R^n$ is
given by\footnote{Constants $C_n$ are allowed to change from line to
line.}
\[
R_0^+(\lambda^2)(x,y) = C_n\,
\frac{\lambda^\frac{n-2}{2}}{|x-y|^{\frac{n-2}{2}}}
H_{\frac{n-2}{2}}^+(\lambda|x-y|)
\]
where $H_\nu^+$ is a Hankel function. There is the scaling relation
\begin{equation}   \label{eq:Rscale}
R_0^+(\lambda^2)(x,y)= \lambda^{n-2} R_0^+(1)(\lambda x,\lambda y)
\quad \forall\;\lambda>0
\end{equation}
and the representation, see the asymptotics of $H_\nu^+$
in~\cite{AbSteg},
\begin{equation}
  \label{eq:Hsplit}
  R_0^+(1) (x,y) = \frac{e^{i|x-y|}}{|x-y|^{\frac{n-1}{2}}}
   a(|x-y|) + \frac{b(|x-y|)}{|x-y|^{n-2}}
\end{equation}
provided $n\ge3$ where
\begin{equation}\label{eq:a}
 |a^{(k)}(r)| \les r^{-k} \quad\forall\;k\ge0, \quad a(r)=0 \quad\forall\; 0< r<1
 \end{equation}
and $b(r)=0$ for all $r>2$, with
\begin{align}
&  |b^{(k)}(r)| \les 1 \quad\forall\;k\ge0, \quad n\text{\ \
odd} \label{eq:bodd}\\
&\left. \begin{array}{lll} |b^{(k)}(r)| &\les 1 &
\forall\;0\le k< n-2 \\
|b^{(k)}(r)| &\les r^{n-k-2}|\log r| &\forall\;k\ge n-2
\end{array}\right\}
 \; n\ge 4\text{\ \ even} \label{eq:beven}
\end{align}
for all $r>0$. As in Chapter~XIV of~\cite{Hor} define
\[
\|f\|_{B}:=\sum_{j=0}^\infty 2^{\frac{j}{2}} \|f\|_{L^2(D_j)}, \quad
\|f\|_{B^*} := \sup_{j\ge0} 2^{-\frac{j}{2}} \|f\|_{L^2(D_j)}
\]
where $D_j=\{ x\::\: |x|\sim 2^j\}$ for $j\ge1$ and $D_0=\{|x|\le
1\}$.

\begin{lemma}
  \label{lem:Bscale} For any $\lambda\ge 1$,
  \[
\|f( \lambda^{-1}\cdot) \|_B \les \lambda^{\frac{n+1}{2}}
\|f\|_B,\quad \|g(\lambda \cdot) \|_{B^*} \les
\lambda^{-\frac{n-1}{2}} \|g\|_{B^*}
  \]
provided the right-hand sides are finite.
\end{lemma}
\begin{proof} By duality, it suffices to prove the first estimate.
 Assume without loss of generality that $\lambda=2^N$ for some
$N\ge0$. Then
\begin{align*}
\|f( \lambda^{-1}\cdot) \|_B &\les  \sum_{j=N}^\infty
2^{\frac{j}{2}} 2^{\frac{nN}{2}} \|f\|_{L^2(D_{j-N})} +
2^{\frac{N}{2}} 2^{\frac{nN}{2}} \|f\|_{L^2(D_0)}  \les
2^{\frac{N(n+1)}{2}} \|f\|_B
\end{align*}
as claimed.
\end{proof}

This lemma and the scaling relation \eqref{eq:Rscale} immediately
imply the following statement. In what follows, $R_0$ stands for
either of $R_0^{\pm}$.

\begin{cor}
  \label{cor:R0scale} If $R_0(1):\: B\to B^*$, then
  \[
\|R_0(\lambda^2)\|_{B\to B^*} \les \lambda^{-1} \|R_0(1)\|_{B\to
B^*}
  \]
  for all $\lambda\ge1$.
\end{cor}
\begin{proof}
  First, from \eqref{eq:Rscale}
  \[
 (R_0^+(\lambda^2) f)(x) = \lambda^{-2} [R_0^+(1)
 f(\cdot\lambda^{-1})] (\lambda x)
  \]
  Hence, by the previous lemma,
  \[
\|R_0^+(\lambda^2) f\|_{B^*} \les \lambda^{-2}
\lambda^{-\frac{n-1}{2}} \| R_0^+(1)
 f(\cdot\lambda^{-1}) \|_{B^*}
 \les \lambda^{-1} \|R_0^+(1)f\|_{B^*}
  \]
  as claimed.
\end{proof}

For any $\delta\in(0,1)$, let $\Phi_\delta$ be a smooth cut-off
function to a $\delta$-neighborhood of the north pole in $S^{n-1}$.
Also, for any $d\in(0,\infty)$,  $\eta_d(x)=\eta(|x|/d)$ denotes a smooth
cut-off to the set $|x|>d$. In what follows, we shall use the
notation
\[
\calR_{d,\delta}(\lambda^2)(x,y)=[R_0(\lambda^2) \eta_d \Phi_\delta]
(x,y) = R_0(\lambda^2)(x,y) \eta_d(|x-y|)
\Phi_\delta\Big(\frac{x-y}{|x-y|}\Big)
\]
Note that this operator obeys the same scaling as $R_0$,
see~\eqref{eq:Rscale}. More precisely,
\[
\calR_{d,\delta}(\lambda^2)(x,y) = \lambda^{n-2}
\calR_{d\lambda, \delta}(1)(\lambda x,\lambda y)
\]
Thus, Corollary~\ref{cor:R0scale}  applies to
$\calR_{d,\delta}(\lambda^2)$ in the form
\begin{equation}\label{eq:calRscale}
\|\calR_{d,\delta}(\lambda^2)\|_{B\to B^*} \les \lambda^{-1}
\|\calR_{d\lambda, \delta}(1)\|_{B\to B^*}
\end{equation}
for all $\lambda\ge1$ or, more generally,
\begin{equation}\label{eq:dercalRscale}
\| D^{\alpha}\calR_{d,\delta}(\lambda^2)\|_{B\to B^*} \les
\lambda^{-1+|\alpha|}
\|D^{\alpha}\calR_{d\lambda, \delta}(1)\|_{B\to B^*}
\end{equation}
for all multi-indices $\alpha$ and $\lambda\ge1$.

The main goal of this section is to prove a limiting absorption
bound for $\calR_{d,\delta}$ and its derivatives of order at most
two uniformly in the parameters $d,\delta\in(0,1)$, see
Proposition~\ref{prop:dir_res} below. This will be based on the
oscillatory integral estimate in Lemma~\ref{lem:VCP}. We first state
a simple technical fact which will be used repeatedly.

\begin{lemma}
\label{lem:L1L2} Let $K(x,y)$ be the kernel of the $L^2$ bounded
operator $T: L^2(\R^n)\to L^2(\R^m)$ with
\[ (Tf)(x) = \int_{\R^n} K(x,y) f(y)\, dy \]
Let $L_1:\R^n\to\R^n$ and $L_2:\R^m\to \R^m$ be invertible linear
transformations and define
\[ (\tilde T f)(x) = \int_{\R^n} K(L_1 x, L_2y) f(y)\, dy \]
Then
\[ \sqrt{|\det L_1||\det L_2|}\; \|\tilde T\|_{2\to2} =  \|T\|_{2\to2}
\]
\end{lemma}

The following lemma is the main technical tool of this section.

\begin{lemma}
  \label{lem:VCP} Let $\chi$  denote a
 smooth cut-off function to the
  region $1<|x|<2$. With $a(r)$ as in~\eqref{eq:a}, define
\begin{equation}\label{eq:pint}
(T_{\delta,p, R_1, R_2} f)(x) = \int \chi\big(\frac{x}{R_1}\big)
\frac{e^{i|x-y|}}{|x-y|^p}
   a(|x-y|)
  \Phi_\delta\Big(\frac{x-y}{|x-y|}\Big) \chi\big(\frac{y}{R_2}\big) f(y)\, dy
\end{equation}
Then, for any $n\ge3$, and
  $\frac{n-1}{2}\le p\le \frac{n+3}{2}$,
\begin{equation}\label{eq:Tbound}
\| T_{\delta,p, R_1, R_2} f \|_2 \le C_n\, \delta^{p-\frac{n-1}{2}}
\sqrt{R_1 R_2}\, \|f\|_2
\end{equation}
  for all $R_1, R_2\ge1$, $\delta\in(0,1)$. The constant $C_n$ only
  depends on $n\ge3$.
\end{lemma}
\begin{proof}
We first consider the cases where $R_2
> 4R_1$ or $R_1>4R_2$. By duality it suffices to treat the first
case. We then distinguish two further cases, depending on whether
$\delta R_2>R_1$ or not.

\medskip {\em Case 1:} $\delta R_2>R_1$

\noindent On the support of the integrand in~\eqref{eq:pint}, we
have
\[
|y'| \les \delta R_2, \; y_n\sim R_2, \; |x|\les R_1
\]
where $y=(y',y_n)$.  By the change of variables $y=(\delta R_2 v',
R_2 v_n)$ and $x=R_1 u$ and Lemma~\ref{lem:L1L2},
\begin{equation}\label{eq:h}
\begin{split}
&\| T_{\delta,p,R_1,R_2}\|_{2\to2} \\
&= \Big\|e^{iR_2\big|\frac{R_1}{R_2} u- ( \delta v', v_n)\big|}\,\, \frac{\chi(u)\chi(\delta
v',v_n)(a\Phi_\delta)(u,v) \, }{\big|\frac{R_1}{R_2} u - (\delta v', v_n)\big|^p}\Big\|_{L^2_v\to L^2_u}
R_1^{\frac{n}{2}} R_2^{\frac{n}{2}-p} \delta^{\frac{n-1}{2}}
\end{split}
\end{equation}
where
\[
(a\Phi_\delta)(u,v) = a(|R_1 u- (R_2 \delta v', R_2 v_n)|) \Phi_\delta \Big( \frac{R_1 u- (R_2 \delta v', R_2
v_n)}{|R_1 u- (R_2 \delta v', R_2 v_n)|}\Big).
\]
We will apply Proposition~\ref{prop:vcp} to the operator in \eqref{eq:h}. First note that the derivatives of
\[
\frac{\chi(u)\chi(\delta v',v_n)(a\Phi_\delta)(u,v)}{\big|\frac{R_1}{R_2} u - (\delta v', v_n)\big|^{p}}
\]
in $u'$ are bounded using the property that $\frac{R_1}{\delta R_2}\les 1$, the symbol-like decay of $a$, the
bounds $|D^\alpha \Phi_\delta|\les \delta^{-|\alpha|}$ and the bound
\begin{equation}
\label{eq:unvn} \big|v_n - \frac{R_1}{R_2} u_n \big| \sim 1.
\end{equation}
Second, the phase $\Psi(u,v)= \big|\frac{R_1}{R_2} u- ( \delta v', v_n)\big|$ satisfies the hypothesis of
Proposition~\ref{prop:vcp}. Indeed,
\[
\nabla_{u'}\Psi(u',u_n,v',v_n) = \frac{R_1}{R_2}\frac{(\frac{R_1}{R_2} u' - \delta v',0)}{|\frac{R_1}{R_2} u- (
\delta v', v_n)\big|} = \frac{R_1}{R_2}\frac{(\frac{R_1}{R_2} u' - \delta v',0)}{|(\frac{R_1}{R_2} u'- \delta
v', \frac{R_1}{R_2}u_n - v_n)\big|}
\]
so that \begin{align*} &\nabla_{u'}\Psi(u',u_n,v',v_n) -
\nabla_{u'}\Psi(u',u_n,w',w_n) \\
&= \frac{R_1}{R_2}\frac{(\frac{R_1}{R_2} u' - \delta v',0)}{|(\frac{R_1}{R_2} u'- \delta v', \frac{R_1}{R_2}u_n
- v_n)\big|} - \frac{R_1}{R_2}\frac{(\frac{R_1}{R_2} u' - \delta w',0)}{|(\frac{R_1}{R_2} u'- \delta w',
\frac{R_1}{R_2}u_n - w_n)\big|}
\end{align*}
Now observe the following: if $x,y\in \R^{k}$, satisfy $|x|,|y|\ll 1$, then
\[\begin{split}
&\Big|\frac{x}{\sqrt{1+|x|^2}} - \frac{y}{\sqrt{1+|y|^2}} \Big| =\frac{|x-y|}{\sqrt{1+|x|^2}} + |y| \,
O\Big(\frac{1}{\sqrt{1+|x|^2}} - \frac{1}{\sqrt{1+|y|^2}} \Big) \\& \sim |x-y|
\end{split}\]
Thus, in view of \eqref{eq:unvn},
\[
|\nabla_{u'}\Psi(u',u_n,v',v_n) - \nabla_{u'}\Psi(u',u_n,w',w_n)| \sim \frac{R_1}{R_2}\delta |v'-w'|
\]
as desired. Moreover, the higher derivatives satisfy
\[
\Big|D^{\beta}_{u'}\Big[\nabla_{u'}\Psi(u',u_n,v',v_n) - \nabla_{u'}\Psi(u',u_n,w',w_n)\Big]\Big| \les
\frac{R_1}{R_2}\delta |v'-w'|
\]
for any $\beta$. In fact, we gain factors of $\frac{R_1}{R_2}$ for the higher derivatives, but this is of no use
to us. Thus, we apply Proposition~\ref{prop:vcp} with $\lambda=R_2$, $\mu= \frac{R_1}{R_2}\delta$, $n_1=n-1$ to
obtain
\[
\| T_{\delta,p,R_1,R_2}\|_{2\to2} \les R_1^{\frac{n}{2}}
R_2^{\frac{n}{2}-p} \delta^{\frac{n-1}{2}} (\delta
R_1)^{-\frac{n-1}{2}}\les \sqrt{R_1R_2}R_2^{\frac{n-1}{2}-p}
\]
which implies the stated bound since $R_2^{-1}\leq \delta/R_1 \leq \delta$.

\medskip {\em Case 2:} $\delta R_2\le  R_1$

Let $\eta$ be a smooth bump function supported in a neighborhood of
the origin such that it defines a partition of unity of $\R^{n-1}$
via
\[
\sum_{k'\in \Z^{n-1}} \eta(x'-k')=1 \quad \forall\; x'\in \R^{n-1}
\]
so that also
\[\sum_{k'\in \Z^{n-1}} \eta\Big(\frac{x'-\delta R_2
k'}{\delta R_2}\Big)=1
\]
This latter partition of unity induces a partition of the $x$ and
$y$ supports in~\eqref{eq:pint} into cylinders of dimensions $\delta
R_2 \times\ldots\times \delta R_2 \times R_1$, and $\delta R_2\times
\ldots\times \delta R_2 \times R_2$, respectively. If $x$ belongs to
a fixed cylinder, then $\Phi_\delta(x,y)\ne0$ implies that $y$
belongs to a finite number of adjacent cylinders, and this number is
uniformly controlled. By almost orthogonality, it suffices to prove
the desired bound for the kernel localized to such cylinders.  After
a translation we can assume that the cylinders are
\[
|x'|\les R_2\delta, \quad |x_n|\les R_1, \qquad |y'|\les
R_2\delta,\quad y_n\sim R_2
\]
Let
\[
(x',x_n) = (R_2\delta u', R_1 u_n), \quad (y',y_n) = (R_2\delta v',
R_2v_n)
\]
By Lemma~\ref{lem:L1L2},
\begin{align*}\begin{split}
& \|T_{\delta,p,R_1,R_2}\|_{2\to2}\\
&\les \delta^{n-1} R_1^{\frac12}R_2^{n-\frac12-p} \Big\|e^{iR_2|(
\delta u',
  \frac{R_1}{R_2} u_n)-(\delta v',v_n)|}
\,\, \frac{\chi(u)\chi(v)(a\Phi_\delta)(u,v)}{\big|
 ( \delta u',
  \frac{R_1}{R_2} u_n\big)-(\delta v',v_n) \big|^{p}}\Big\|_{L_v^2\to L^2_u}
\end{split}\end{align*} where
\begin{equation}\label{eq:aPhishit}\begin{split}
(a\Phi_\delta)(u,v) &= a(|( R_2\delta u',
  R_1 u_n)-(R_2\delta v', R_2v_n)| )\times\\
  &\qquad \times \Phi_\delta \Big(\frac{(R_2 \delta u',
  {R_1} u_n)-(R_2\delta v', R_2v_n) }{ |( R_2\delta u',
  {R_1} u_n)-(R_2\delta v', R_2 v_n)|  } \Big)
  \end{split}
\end{equation}
 On the support of the integrand, $|u|, |v|\les 1 $, and
$v_n\sim1$.  Here the kernel is bounded in absolute value by
$(\frac43)^p \chi(u)\chi(v)$ since $a\Phi_\delta$ is bounded, $v_n \sim 1$,
and $\frac{R_1}{R_2} < \frac14$.  Schur's test gives the immediate bound
\[\|T_{\delta,p,R_1,R_2}\|_{2\to2}
  \les  R_1^{\half}R_2^{n-\half-p}\delta^{n-1}
\]
If $R_2\delta^2 \le 1$, this estimate is sufficient because
$R_2^{n-1-p}\delta^{n-1} \le \delta^{p-\frac{n-1}{2}}$.  The last inequality
is verified in two ways: if
$\frac{n-1}{2} \le p \le n-1$, one can write
\[
R_2^{n-1-p}\delta^{n-1} \le \delta^{-2(n-1-p)} \delta^{n-1}\le
\delta^{2p-(n-1)} \le \delta^{p-\frac{n-1}{2}}
\]
On the other hand, if $p>n-1$, then
\[
R_2^{n-1-p}\delta^{n-1} \le  \delta^{n-1}\le
\delta^{p-\frac{n-1}{2}}
\]
since $p\le \frac{n+3}{2}\le 3\frac{n-1}{2}$ when $n\ge3$.

When $R_2\delta^2 \ge1$, an improved operator estimate can be obtained
via Proposition~\ref{prop:vcp}.
Observe that the $u'$-derivatives of~\eqref{eq:aPhishit} are uniformly bounded. Furthermore, the
same analysis as in the previous case applies to the  phase  \[\Psi(u,v)=|( \delta u',
  \frac{R_1}{R_2} u_n)-(\delta v',v_n)|\]
  with $\mu=\delta^2$, $\lambda=R_2$,
since we still have $|u|\les 1$, $|v|\les 1$, as well as $v_n\sim1$.
Proposition~\ref{prop:vcp} now provides the desired estimate
\[\begin{split}
\|T_{\delta,p,R_1,R_2}\|_{2\to2} &\les \delta^{n-1}
R_1^{\frac12}R_2^{n-\frac12-p} (R_2\delta^2)^{-\frac{n-1}{2}} \\
&= R_2^{\frac{n-1}{2}-p}\sqrt{R_1R_2} \les \delta^{p-\frac{n-1}{2}}
\sqrt{R_1R_2} \end{split}\]
where we have used the condition $p \ge \frac{n-1}{2}$ twice in the last line.

\medskip
Finally, we need to consider the case $R_1\sim R_2\sim R$ where
$R\ge 1$. Let $\|f\|_2\le1$. Then
\begin{align}
  &\| T_{\delta,p,R_1,R_2} f \|_{2} \nonumber\\
& \le \sum_{1\le 2^j\le R} \Big\| \int
e^{iR|x-y|}\,\frac{\chi_j(x,y)
  (a\Phi_\delta)(Rx,Ry) }{|x-y|^p}\,f(Ry)\, dy \Big\|_{L^2_x}
  R^{\frac{3n}{2}-p}\label{eq:2j}
\end{align}
where $\chi_j(x,y) $ is a smooth cut-off function on the set
$\{(x,y):|x|,|y|<1, |x-y|\sim2^{-j}\}$. Performing a Whitney
decomposition of the integrand away from the diagonal $x=y$, we can
estimate~\eqref{eq:2j} by
\begin{align} R^{n-p}\sum_{1\le 2^j\le R} \max_{Q_1^{(j)}\sim
Q_2^{(j)}} \Big\| \chi_{Q_1^{(j)}}(x)
  \frac{e^{iR|x-y|}}{|x-y|^p} \chi_{Q_2^{(j)}}(y)
  (a\Phi_\delta)(Rx,Ry) \Big\|_{L^2_y\to L_x^2} \label{eq:Q1Q2}
  \end{align}
Thus,  $Q_1^{(j)},  Q_2^{(j)}$ are cubes of side length $2^{-j}$ and
$Q_1^{(j)}\sim Q_2^{(j)}$ denotes that they are "related", i.e.,
$\dist(Q_1^{(j)},Q_2^{(j)})\sim 2^{-j}$. Now fix $j$ and cubes
$Q=Q_1^{(j)},Q'= Q_2^{(j)}$. We break $Q$ and $Q'$ into cylinders of
size $2^{-j}\delta\times\ldots \times 2^{-j}\delta \times 2^{-j}$.
Because of the directional cut-off $\Phi_\delta$, each $Q$ cylinder
interacts with at most finitely many $Q'$ cylinders. For one such
pair of cylinders, we can assume (after translation) that
\[
x=(2^{-j}\delta u', 2^{-j} u_n),\qquad y= (2^{-j}\delta v', 2^{-j}
v_n)
\]
where $|u|,|v|\les 1$, $v_n-u_n\sim 1$. By Lemma~\ref{lem:L1L2}
\begin{align} & \Big\| \chi_{Q}(x)
  \frac{e^{iR|x-y|}}{|x-y|^p}\; \chi_{Q'}(y)
  (a\Phi_\delta)(Rx,Ry) \Big\|_{L^2_y\to L_x^2} \nonumber\\
 & \les 2^{j(p-n)} \delta^{n-1} \Big\| e^{iR2^{-j} | (\delta u', u_n) - (\delta v',
 v_n)|}\;
  \frac{\chi(u)\chi(v) (a\Phi_\delta)(u,v)}{|(\delta
  u',u_n)-(\delta v',v_n)|^p}\Big\|_{L_v^2\to L_u^2} \nonumber\\
  &\les 2^{j(p-n)} \delta^{n-1} \min\big(1, (R\delta^2
  2^{-j})^{-\frac{n-1}{2}}\big) \label{eq:finalclaim}
\end{align}
where \[\begin{aligned} (a\Phi_\delta)(u,v) &= a(R2^{-j}|(\delta u',
  u_n)-(\delta v', v_n)| ) \Phi_\delta \Big(\frac{( \delta u',
   u_n)-(\delta v', v_n) }{ |( \delta u',
   u_n)-(\delta v', v_n)|  } \Big)
  \end{aligned}
\]
\eqref{eq:finalclaim} follows from Schur's test and Proposition~\ref{prop:vcp}. For the latter note that the
$u'$ derivatives of $(a\Phi_\delta)(u,v)$ are uniformly bounded on the support of the integrand. Furthermore,
the phase is $\Psi(u,v) = | (\delta u', u_n) - (\delta v', v_n)|$ and we have $|u|,|v|\les 1$, $v_n-u_n\sim 1$.
Thus, as in the previous cases, the proposition applies with $\mu=\delta^2$, $\lambda=R2^{-j}$.

Combining \eqref{eq:2j}, \eqref{eq:Q1Q2}, and \eqref{eq:finalclaim},
yields
\begin{equation}
\|T_{\delta,p,R_1,R_2}\|_{2\to2} \les \sum_{1\le 2^j\le R} R^{n-p}
2^{j(p-n)} \delta^{n-1} \min\big(1, (R\delta^2
  2^{-j})^{-\frac{n-1}{2}}\big) \label{eq:sumj}
\end{equation}
Note that $p<n$ unless $n=3=p$. In that case the right-hand side of
\eqref{eq:sumj} is $\les \delta^2\log R \les R\delta^2$. For the
remainder of the proof, therefore, we may assume $p<n$. First consider the case
$R\delta^2\le 1$ where we have
\[
\eqref{eq:sumj} \les R^{n-p}  \delta^{n-1} = R
\delta^{p-\frac{n-1}{2}} R^{n-1-p}\delta^{3\frac{n-1}{2}-p}\les R
\delta^{p-\frac{n-1}{2}}
\]
To prove the final inequality distinguish the cases $p\ge n-1$ and
$p<n-1$ and note that $p\le \frac{n+3}{2}\le 3\frac{n-1}{2}$.
Henceforth $R\delta^2>1$ and we  distinguish between $R\delta^2 \le
2^j\le R$ and $1\le 2^j\le R\delta^2$. The contribution to the sum
in~\eqref{eq:sumj} by the former is
\[
  R^{n-p} \delta^{n-1} (R\delta^2)^{p-n} = \delta^{2p-(n+1)} = R(R\delta^2)^{-1} \delta^{2(p-(n-1)/2)} \les
R\delta^{p-\frac{n-1}{2}}
\]
since $p\ge \frac{n-1}{2}$. The contribution by $1\le 2^j\le
R\delta^2$ to~\eqref{eq:sumj} is
\begin{equation}
  R^{\frac{n+1}{2}-p}  \sum_{1\le 2^j \le \delta^2 R}
2^{-j(\frac{n+1}{2}-p)} \label{eq:sum2}
\end{equation}
If $\frac{n-1}{2}\le p<\frac{n+1}{2}$, then
\[
\eqref{eq:sum2} \les R^{\frac{n+1}{2}-p} \les R
\delta^{p-\frac{n-1}{2}}
\]
since $R\delta\ge R\delta^2\ge1$. If $p=\frac{n+1}{2}$, then
\[
\eqref{eq:sum2} \les  \log (R\delta^2) \les R \delta^2 \les R\delta
\]
since again $R\delta^2\ge1$. Finally, if $p>\frac{n+1}{2}$, then
\[
\eqref{eq:sum2} \les  \delta^{2p-n-1} \les R\delta^{p-\frac{n-1}{2}}
\]
which concludes the proof.
\end{proof}

\begin{prop}
  \label{prop:dir_res}
  Let $n\ge3$. Then for any $d\in (0,\infty)$, $\delta\in (0,1)$, and
$\lambda\ge 1$ there is the bound
\begin{equation}\label{eq:limap}
\| D^{\alpha}\calR_{d,\delta}(\lambda^2)f \|_{B^*} \le
C_n\lambda^{-1+|\alpha|} \|f\|_B
\end{equation}
for any $0\le|\alpha|\le 2$. The constant $C_n$ depends only on the
dimension $n\ge3$.
\end{prop}
\begin{proof} In view of \eqref{eq:calRscale} and \eqref{eq:dercalRscale}
it suffices to prove these estimates for $\lambda=1$.  We need to
prove that for any $0\le|\alpha|\le 2$
\begin{equation}\label{eq:star}
\| \chi(\cdot/R_1) D^\alpha \calR_{d,\delta}(1)\, \chi(\cdot/R_2) f
\|_2 \le C_n\, \sqrt{R_1 R_2}\, \|f\|_2
\end{equation}
where $R_1, R_2\ge 1$ are arbitrary.  We write
\begin{equation}\label{eq:T0T1}
\calR_{d,\delta}(1)=R_0^+(1) \eta_d\Phi_\delta   = T_0 + T_1
\end{equation}
where the kernels of $T_0, T_1$ are
\begin{equation}\label{eq:Tdef}\begin{split}
 T_0(x,y) &=
 \frac{b(|x-y|)}{|x-y|^{n-2}}
   \eta_d(|x-y|) \Phi_\delta(x,y)\\
T_1(x,y) &= \frac{e^{i|x-y|}}{|x-y|^{\frac{n-1}{2}}}
   \eta_d(|x-y|)a(|x-y|)
  \Phi_\delta(x,y),
  \end{split}
\end{equation}
respectively, see \eqref{eq:Hsplit}. The modified function $\eta_d(r)a(r)$
satisfies all decay estimates in~\eqref{eq:a} with constants independent of
the choice of $d$.

We begin by showing that
$\widehat{T_0 f} = m_0 \hat{f}$ where $|m_0(\xi)|\les \la
\xi\ra^{-2}$. This will imply~\eqref{eq:star} for~$T_0$. By
definition
\[
m_0(\xi) = \int_0^\infty\int_{S^{n-1}} rb(r)  \eta_d(r) e^{-ir
\omega\cdot\xi} \Phi_\delta(\omega) \,\sigma(d\omega) \,dr
\]
Since $b(r)=0$ if $r>2$, $|m_0(\xi)|\les 1$. Hence we may assume
that $|\xi|\ge1$.  If $|\xi_n| \ge |\xi|/10$, then $|\omega \cdot \xi| \ges
|\xi|$ and
\[
|m_0(\xi)|\les \int_{S^{n-1}} \Phi_\delta(\omega) \la
\omega\cdot\xi\ra^{-2}\,\sigma(d\omega)\les \delta^{n-1} |\xi|^{-2}
\]
where we have used that
\[
\Big| \int_0^\infty e^{-ir\rho} rb(r)\eta_d(r)\chi(r)\,dr\Big |\les
\la \rho\ra^{-2}
\]
This follows from \eqref{eq:bodd} and \eqref{eq:beven} after two
integrations by parts. Now suppose that $|\xi_n| \le
|\xi|/10$. Set $\xi=|\xi| \hat{\xi}$ and change integration
variables as follows:
\begin{align*}
&\int_{S^{n-1}}\int_0^\infty rb(r)  \eta_d(r) \chi(r) e^{-ir|\xi|
\omega\cdot\hat{\xi}}\,dr\, \Phi_\delta(\omega) \,\sigma(d\omega)\\
& = \int_{\R^{n-1}} \int_0^\infty rb(r)  \eta_d(r) \chi(r)
e^{-ir|\xi| u_1}\,dr \,\tilde \Phi_\delta(u_1,\ldots,u_{n-1})\,
du_1du_2\ldots du_{n-1}\\
& = \delta^{n-2} \int_0^\infty \int_\R rb(r)  \eta_d(r) \chi(r)
e^{-ir|\xi| u_1}\Psi_\delta(u_1)\, du_1 dr,
\end{align*}
where $(u_1, \ldots, u_{n-1})$ is a parametrization of the support of
$\Phi_\delta$, aligning $u_1$ with $\hat{\xi}$.  The function $\Psi_\delta$
is a smooth cut-off supported on an interval of
length $\sim \delta$ resulting from the integration of $\tilde{\Phi}_\delta$.
Thus,
\[
|m_0(\xi)|\les \delta^{n-2}\int_0^1
r|\widehat{\Psi_\delta}(r|\xi|)|\,dr\les
\delta^{n-2}|\xi|^{-2}\|u\widehat{\Psi_\delta}(u)\|_{L^1_u}\les
\delta^{n-3}|\xi|^{-2}.
\]
In conclusion, $|m_0(\xi)|\les \la \xi\ra^{-2}$ as claimed.

Next, consider $T_1$.  By the Leibniz rule,
\begin{align}
D^\alpha_x T_1(x,y) &= \sum_{\beta\le \alpha} c_{\alpha,\beta}\,
D^{\alpha-\beta}_x \Big[\frac{e^{i|x-y|}}{|x-y|^{\frac{n-1}{2}}}
  \eta_d(|x-y|) a(|x-y|)\Big] D^{\beta}_x
  \Phi_\delta(x,y) \nonumber\\
  & = \sum_{\beta\le \alpha} \delta^{-|\beta|}\,c_{\alpha,\beta}\,
\frac{e^{i|x-y|}}{|x-y|^{\frac{n-1}{2}+|\beta|}}
  a_{\alpha,\beta, d}(|x-y|)
  \Phi_{\delta,\beta}(x,y)\label{eq:newT1}
\end{align}
where $\Phi_{\delta,\beta}= \delta^{|\beta|} D^\beta \Phi_\delta$ is a
modified angular cut-off and $a_{\alpha,\beta, d}$ satisfies the same bounds
as $a$, see~\eqref{eq:a}, with constants that do not depend on $d$.
The estimate~\eqref{eq:star} for $T_1$ follows from Lemma~\ref{lem:VCP}
with $p=\frac{n-1}{2}+|\beta|$.
\end{proof}

For any $\lambda\ge1$ define
\begin{align*} X^*_\lambda &:= \{ f\in
B^*\::\: \la \nabla\ra f\in B^*\} \\
 \|f\|_{X^*_\lambda} &:=
\|f\|_{B^*} + \lambda^{-1} \|\la \nabla\ra f\|_{B^*}
\end{align*}
The dual norm is
\[
\|f\|_{X_\lambda} := \inf_{f=f_1+f_2} \Big(\|f_1\|_{B} + \lambda
\|\la \nabla\ra^{-1} f_2\|_{B}\Big)
\]

\begin{cor} \label{cor:R_0}
Let $\calR_{d,\delta}$ be as above.  Then for all $\lambda \ge 1$
\begin{equation} \label{eq:R_0}
\| \calR_{d,\delta}(\lambda^2)f\|_{X^*_\lambda} \le C_n \lambda^{-1}
\|f\|_{X_\lambda}
\end{equation}
uniformly in $d \in (0,\infty)$, $\delta\in [0,1]$.
\end{cor}
\begin{proof}
 This follows from Proposition~\ref{prop:dir_res} provided the
 estimate
 \[
\|\la \nabla\ra f\|_{B^*} \les \|f\|_{B^*} + \|\nabla f\|_{B^*}
 \]
 holds. This in turn will follow if we can show that
 $\|(m\hat{f})^{\vee}\|_{B^*}\les \|f\|_{B^*}$ for any
 symbol $m$ with bounded derivatives. However, this is
 guaranteed by Corollary~14.1.5 in~\cite{Hor}.
\end{proof}

\section{The high energies limiting absorption principle}
\label{sec:high_limap}

The main result of this section is a limiting absorption principle
for the perturbed resolvent
\begin{equation}\label{eq:res_iden}
R_L^+(\lambda^2) = (I + R_0^+(\lambda^2)L)^{-1} R_0^+(\lambda^2)
\end{equation}
where $L = i(\nabla \cdot A + A \cdot \nabla) + V$, see
Proposition~\ref{prop:highen} below. As before, we shall mostly drop
the superscript $+$ on the resolvent. We shall assume throughout
this section that $A,V\in Y$ where
\[ Y:=\Big\{ f\in L^\infty\::\: \sum_{j=0}^\infty 2^j \|f\|_{L^\infty(D_j)}<\infty\Big\}
\]
This is the space of functions that take $B^*\to B$ by
multiplication.

\begin{lemma} \label{lem:L}
For any $\lambda\ge1$
\begin{equation} \label{eq:L}
\norm[Lf][X_\lambda] \le C_n (\lambda \|A\|_Y +
\|V\|_Y)\norm[f][X^*_\lambda]
\end{equation}
\end{lemma}
\begin{proof}
Multiplication by $V$ is bounded $B^*\to B$. Also,
\[
\| \la \nabla\ra^{-1} \nabla A f\|_{B} \les \| Af\|_{B}\le
\|A\|_Y\|f\|_{B^*}
\]
where the first inequality follows  from Corollary~14.1.5
in~\cite{Hor}. Hence $\nabla\cdot A:X^*_\lambda\to X_\lambda$ with
norm $\les\lambda$. By duality the same holds for $A\cdot\nabla$.
\end{proof}

From this and Corollary~\ref{cor:R_0} it follows that
$$\norm[(I + R_0(\lambda^2)L)^{-1}f][X^*_\lambda] \le 2 \norm[f][X^*_\lambda]$$
for all $\lambda\ge 1$ provided $A$ is small in~$Y$.

The main goal of this section is to show that even when $A$ is not
small the Neumann series
\begin{equation}\label{eq:geomser}
(I + R_0(\lambda^2)L)^{-1} = \sum_{\ell=0}^\infty (-1)^\ell
(R_0(\lambda^2)L)^\ell
\end{equation}
converges for large $\lambda$. This cannot be deduced from the size
of $R_0(\lambda^2)L$ alone, but is instead a consequence of the
following crucial lemma.

\begin{lemma} \label{lem:largem}
Assume that $A, V \in Y$, with $A$ also being continuous.
Given any constant $c>0$, there
exist sufficiently large $m = m(c,A,V)$ and $\lambda_1 =
\lambda_1(c,A,V)$ such that
\begin{equation}  \label{eq:smallprod}
  \sup_{\lambda>\lambda_1}\| (R_0(\lambda^2) L)^m \|_{X^*_\lambda \to X^*_\lambda} \le c
\end{equation}
More generally, given any $r > 0$, there exist sufficiently large $m = m(r,A,V)$ and
$\lambda_1(r,A,V)$ such that
\[
\sup_{\lambda>\lambda_1}\| (R_0(\lambda^2) L)^m \|_{X^*_\lambda \to X^*_\lambda} \le
(2r)^m \]
\end{lemma}

By choosing $c = \frac12$, the series in~\eqref{eq:geomser} becomes absolutely
convergent.  In view of \eqref{eq:res_iden}, we thus
conclude the following limiting absorption principle for large
energies:

\begin{prop}\label{prop:highen} Under the conditions of the previous
lemma,  there exists $\lambda_1=\lambda_1(A,V)$ so that for all
$\lambda\ge\lambda_1$ one has $R_L(\lambda^2):X_\lambda\to
X^*_\lambda$ with norm estimate
\[ \| R_L(\lambda^2)f\|_{X^*_\lambda} \le C_n\, \lambda^{-1}
\|f\|_{X_\lambda}
\]
for all $\lambda\ge\lambda_1$.
\end{prop}

As a corollary, we obtain the desired $L^2$ bounds on
$Z_0 R_L(\lambda^2)Z_0^*$ as required in Section~\ref{sec:setup}.

\begin{cor}\label{cor:highenZ}
With $\lambda_1$ as above, there are the bounds
\[ \begin{split}
\sup_{\lambda\ge\lambda_1} \big\| \la x\ra^{-\sigma}
\la\nabla\ra^{\half} R_L(\lambda^2+i0)\la\nabla\ra^{\half}\la
x\ra^{-\sigma} \big\|_{2\to2}
 \le C_n  \\
\sup_{\lambda\ge\lambda_1} \big\| \la x\ra^{-\sigma}
|\nabla|^{\half} R_L(\lambda^2+i0) |\nabla|^{\half}\la
x\ra^{-\sigma} \big\|_{2\to2}
 \le C_n
\end{split}
\]
for any $\sigma > \frac12$.  In particular, $Z_0 R_L(\lambda^2)Z_0^*$
is uniformly bounded in~$L^2$ for $\lambda\ge \lambda_1$.
\end{cor}
\begin{proof}
Let $Z= \la x \ra^{-\sigma} \la\nabla\ra^{\half}$.
In view of Proposition~\ref{prop:highen}, in order for $ZR_L(\lambda^2)Z^*$ to
be uniformly bounded in $L^2$, we need to prove that
$Z:X_\lambda^*\to L^2$ with norm $\les \sqrt{\lambda}$, and,
equivalently that $Z^*:L^2\to X_\lambda$ with the same norm.
These estimates follow rather directly from the definition of the space
$X_\lambda^*$.

If $\norm[f][X_\lambda^*] = 1$, then $f \in B^*$ and
$\japanese[x]^{-\sigma}f \in L^2$, each with bounded norm. At the same time,
$\norm[\japanese[x]^{-\sigma}\japanese[\nabla]f][2] \les \lambda$.
By the commutator bound in the appendix, it is possible to interchange the
weight and the derivative.  Therefore by Parseval's identity,
\[
\norm[\japanese[\nabla]^{\frac12}\japanese[x]^{-\sigma}f][2]^2
\ \les \ \norm[\japanese[\nabla]\japanese[x]^{-\sigma}f][2]
         \norm[\japanese[x]^{-\sigma}f][2]
\ \les \ \lambda \norm[f][X_\lambda^*]^2
\]

Once again the weight and fractional derivative can be interchanged to prove
the bound for $Z R_L(\lambda^2)Z^*$.  The bound for $Z_0 R_L(\lambda^2)Z_0^*$
follows immediately because $\japanese[x]^{-\sigma}|\nabla|^{\frac12}
\japanese[\nabla]^{-\frac12}\japanese[x]^{\sigma}$ is a bounded operator on
$L^2$, see Lemma~\ref{lem:frac}.
\end{proof}

The remainder of this section is devoted to the proof of
Lemma~\ref{lem:largem}.  Due to the estimate
\[
\norm[R_0(\lambda^2)Vf][X^*_\lambda] \les  \lambda^{-1}\|V\|_Y
\norm[f][X^*_\lambda].
\]
we can henceforth assume that $L = i(\nabla \cdot A + A \cdot
\nabla)$, with $V \equiv 0$.  A partition of unity $\{\Phi_i\}$
over $S^{n-1}$ induces a directional decomposition of the free
resolvent, namely
\begin{equation} \label{eq:decomp}
R_0(\lambda^2) = \sum_{i} \calR_i(\lambda^2) + R_d(\lambda^2)
\end{equation}
where $\calR_i(\lambda^2) := \calR_{d,\delta}(\lambda^2)$ with
$\Phi_i$ playing the role of $\Phi_\delta$ from the previous
section. Moreover, $R_d(\lambda^2)(x)=(1 - \eta_d(|x|))R_0(\lambda^2)(|x|)$ is
the ``short range piece''.  Heuristically speaking, $R_d(\lambda^2)$ behaves
like $R_0((\lambda+i\frac{d}{\lambda})^2)$ and  should therefore be bounded on $L^2$ with
operator norm $\les \frac{d}{\lambda}$.
The following lemma makes this precise.

\begin{lemma} \label{lem:R_d'}
With $R_d^+(\lambda^2)$ defined as above, the mapping estimate
\begin{equation} \label{eq:R_d}
\|D^\alpha R_d(\lambda^2)f \|_2 \le C_n\, \lambda^{-2+|\alpha|}
 \la d\lambda\ra \|f\|_2
\end{equation}
holds uniformly for every choice of $d \in (0,\infty)$, $0\le|\alpha|\le2$,
and $\lambda \ge 1$.
\end{lemma}
\begin{proof} By the scaling relation \eqref{eq:Rscale}, for any $\alpha$,
\[
\| D^\alpha R_0(\lambda^2) \chi_{[|x|<d]} \|_{2\to 2}= \lambda^{-2+|\alpha|} \|
D^\alpha R_0(1) \chi_{[|x|<\lambda d]} \|_{2\to2}
\]
where $\chi_{[|x|<\rho]}=\chi(|x|/\rho)$ is a smooth cut-off to the
set $|x|<\rho$ with $\rho>0$ arbitrary.  The notation is somewhat
ambiguous here; we are seeking an estimate for the convolution
operator with kernel $D^\alpha R_0(1)\chi_{[|x| < \lambda d]}$.  The
lemma is proved by showing that the Fourier transform of
$R_0(1)\chi_{[|x| < \rho]}$ is bounded point-wise by
$\japanese[\rho]\japanese[\xi]^{-2}$.

Consider first the case $\rho \le 1$.  The decomposition~\eqref{eq:Hsplit}
implies that $\int_{\R^n} |R_0(1)(x)\chi(|x|/\rho)|\,dx \les \rho^2$.
Furthermore, since $(\Delta+1)R_0(1)$ is a point mass at the origin,
the distribution $\Delta[R_0(1)\chi_{[|x| < \rho]}]$ consists of a point mass
plus a function of bounded $L^1$ norm.  The desired Fourier transform
estimates follow immediately.

When $\rho > 1$, it is more convenient to estimate
\[
\rho^n \Big| \int \big[\PV \frac{1}{|\eta|^2-1} +
i \sigma_{S^{n-1}}(d\eta)\big]\hat{\chi}((\xi-\eta)\rho)\, d\eta  \Big|
\]
A standard calculation shows this to be less than
$\rho\japanese[\rho(|\xi|^2 - 1)]^{-1} < \rho\japanese[\xi]^{-2}$.
\end{proof}

We shall use Lemma~\ref{lem:R_d'} in the following somewhat less precise form:

\begin{lemma}\label{lem:R_d}
For any $0<d<1$
\[
\| R_d(\lambda^2) f\|_{X_\lambda^*} \le C_n\, \lambda^{-1}d \|f\|_{X_\lambda}
\]
uniformly in $\lambda\ge d^{-1}$.
\end{lemma}
\begin{proof}
In view of the definition of the spaces $X_\lambda$, $X_\lambda^*$ this follows from Lemma~\ref{lem:R_d'}
via the imbedding $B\to L^2 \to B^*$ and the identity
$\lambda^{-1}\la d\lambda\ra \sim d$ for $\lambda > d^{-1}$.
\end{proof}

Decomposing each free resolvent in the $m$-fold product
$(R_0(\lambda^2)L)^m$ as in \eqref{eq:decomp} yields the identity
\begin{equation} \label{eq:decomp2}
(R_0(\lambda^2)L)^m = \sum_{i_1 \ldots i_m} \prod_{k=1}^m
  \big(R_{i_k}(\lambda^2)L\big).
\end{equation}
The indices $i_k$ may take numerical values corresponding to the
partition of unity $\{\Phi_i\}$, or else the letter $d$ to indicate
a short-range resolvent. There are two main types of products
represented here, namely:
\begin{itemize}
\item {\em Directed Products}, where the support of functions $\Phi_{i_k}$
and $\Phi_{i_{k+1}}$ are separated by less than $10\delta$ for each
$k$. A product is also considered to be directed if it has this
property once all instances of $i_k = d$ are removed.  The term
$(R_d(\lambda^2)L)^m$ is a vacuous example of a directed product.

\item All other terms not meeting the above criteria are
 {\em Undirected Products}.  An undirected product must contain two
adjacent numerical indices (i.e., after discarding all instances
where $i_k = d$) for which the corresponding functions $\Phi_i$ have
disjoint support with distance at least $10\delta$ between them.
\end{itemize}

\begin{lemma} \label{lem:combinatorics}
For any $\delta>0$, there exists a partition of unity $\{\Phi_i\}$
with approximately $\delta^{1-n}$ elements, having ${\rm diam}\,{\rm
supp}\,(\Phi_i) < \delta$ for each $i$ and admitting no more than
$\delta^{1-n}(C_n)^m$ directed products of length $m$
in~\eqref{eq:decomp2}.
\end{lemma}
\begin{proof}
The first claim is a standard fact from differential geometry. For
the second claim note that there are $\les \delta^{1-n}$ choices for
the first element in a directed product, but only $C_n$ choices at
each subsequent step.
\end{proof}

If $\delta < \frac{1}{20m}$, then a directed product is truly
``directed'' in the sense that all the participating functions
$\Phi_{i_k}$ have support well within a single hemisphere. The
convolution operators $R_{i_k}(\lambda^2)$ are therefore biased
consistently to one side.  In the one-dimensional setting this is
reminiscent of a product of Volterra operators, where a norm
improvement of $m!$ is typical.

\begin{lemma} \label{lem:directedprod}
Suppose $L = i(\nabla \cdot A + A \cdot \nabla)$, with $A(x)\in Y$.
Given any $r > 0$, there exists a distance $d = d(r) > 0$ such that
each directed product in~\eqref{eq:decomp2} satisfies the estimate
\begin{equation}
\bignorm[\prod_{k=1}^m \big(R_{i_k}(\lambda^2)L\big)f][X^*_\lambda]
\le
 C_{n,A,r}\, r^{m} \norm[f][X^*_\lambda]
\end{equation}
uniformly over all $\lambda > d^{-1}$ and all choices of $m$ and $\delta$
satisfying $\delta \le \frac{1}{20m}$.

Consequently, given any $c>0$, there exists a number $m= m(c,A)$ and a
partition of unity governed by $\delta = \frac{1}{20m}$ so that
the sum over all directed products achieves the bound
\begin{equation*} \label{eq:directedprod}
\sum_{\substack{i_1 \ldots i_m \\ directed}} \bignorm[\prod_{k=1}^m
\big(R_{i_k}(\lambda^2)L\big)][X^*_\lambda\to X^*_\lambda] \le
{\textstyle \frac{c}{2}}
\end{equation*}
uniformly in $\lambda > d^{-1}$.
\end{lemma}

\begin{proof}
In this proof, we will keep track of the superscripts
$\pm$ on the resolvents. Also, we will write $\|A\|_Y=C_A$.
There is no loss of generality if we assume that $r < C_nC_A$, where
$C_n$ is the product of the constants in~\eqref{eq:R_0} and~\eqref{eq:L}.

After a rotation, we may assume that every function $\Phi_{i_k}$
which appears in the product has support within a half-radian
neighborhood of the north pole, where $x_n > \frac23$. If $f \in
X_\lambda$ is supported on the half plane $\{x_n
> a\}$, then the support of $R_{i_k}^+(\lambda^2)f$ must be
translated upward to $\{x_n > a + \frac23d\}$.
The short-range resolvent $R_d^+(\lambda^2)$ does
not have a preferred direction; however if $f \in X_\lambda$ is
supported on $\{x_n > a\}$ then ${\rm supp}\,R_d^+(\lambda^2)f
\subset \{x_n > a - 2d\}$.

The purpose of keeping track of supports is that if $f \in
X^*_\lambda$ is supported away from the origin, in the set $\{|x| >
a\}$, then the estimate in Lemma~\ref{lem:L} can be improved to
\begin{equation} \label{eq:distantL}
\norm[Lf][X_\lambda] \les \lambda \norm[A\chi_{[|x| > a]}][Y]
\norm[f][X^*_\lambda],
\end{equation}
since we are assuming that $V \equiv 0$. For $a>0$, the half-plane
$\{x_n > a\}$ is sufficiently far from the origin for this improved
estimate to hold.  Note that the compactly supported functions are dense in
$Y$. Given any $A \in Y$ and any $r>0$, we can choose $R < \infty$ so that
\[\norm[A\chi_{[x_n > R]}][Y] <  \frac{r^2}{C_n^2C_A} \]

Let $\chi$ be a smooth function supported on the interval $[-1,
\infty)$ such that $\chi(x_n) + \chi(-x_n) = 1$.  We will initially
estimate the operator norm of $\big(\prod_k (R_{i_k}^+(\lambda^2)L)
\big) \chi(x_n)$. Multiplication by $\chi(x_n)$ is bounded operator
of approximately unit norm in all spaces $X^*_\lambda$ and
$X_\lambda$.

The support of $\chi(x_n) f$ lies in the half-space $\{x_n > -1\}$.
Suppose every one of the indices $i_k$ is numerical.  Then each
application of an operator $R_{i_k}^+(\lambda^2)L$ translates the
support upward by $\frac23d$. For the first $\frac{3R}{2d}$ steps the
operator norm of $R_{i_k}^+(\lambda^2)L$ is bounded by~\eqref{eq:R_0}
and~\eqref{eq:L}.  Thereafter it is possible to use the stronger bound
of~\eqref{eq:distantL} in place of~\eqref{eq:L} because the support will
have moved into the half-space $\{x_n > R\}$.  The combined estimate is
\begin{equation} \label{eq:alldirected}
\begin{aligned}
\bignorm[\prod_{k=1}^m (R_{i_k}^+(\lambda^2)L)\,\chi(x_n)f][X^*_\lambda]
&\le (C_nC_A)^{m} \Big(\frac{r^2}{(C_nC_A)^2}\Big)^{m-\frac{3R}{2d}}
\norm[f][X^*_\lambda] \\
& = (C_n C_A)^{-m}(r^{-1}C_n C_A)^{\frac{3R}{d}} r^{2m} \norm[f][X^*_\lambda]
\end{aligned}
\end{equation}
This is valid for small $m$ by our assumption that $r < C_nC_A$.

If each directed resolvent $R_{\Phi_i}^+(\lambda^2)$ is seen as
taking one step forward, then the short-range resolvent
$R_d^+(\lambda^2)$ may take as many as three steps back.  Suppose a
directed product includes exactly one index $i_k = d$.  This will
have the most pronounced effect if it occurs near the beginning of
the product, delaying the upward progression of supports by a total
of $4$ steps.  In this case one combines
~\eqref{eq:distantL},~\eqref{eq:R_0}, and Lemma~\ref{lem:R_d}
to obtain
\[
\bignorm[\prod_{k=1}^m (R_{i_k}^+(\lambda^2)L)\,\chi(x_n)f][X^*_\lambda]
\le (C_nC_A)^{m} d \Big(\frac{r^2}{(C_nC_A)^2}\Big)^{m-(\frac{3R}{2d}+4)}
\norm[f][X^*_\lambda]
\]
Notice that this estimate agrees with the one in
\eqref{eq:alldirected} up to
a factor of $d(r^{-1}C_nC_A)^8$.
By setting $d = d(r) = \big(\frac{r}{C_nC_A}\big)^8$,
the bound in \eqref{eq:alldirected} is
strictly larger.  Similar arguments yield the same result for any
directed product with one or more instances of the short-range
resolvent $R_d^+(\lambda^2)$.

To remove the spatial cutoff, write
\begin{equation*}
\prod_{k=1}^m R_{i_k}^+(\lambda^2)L \ = \ \Big(\prod_{k=1}^{m/2}
(R_{i_k}^+(\lambda^2)L)\Big)(\chi(x_n) + \chi(-x_n))
\Big(\prod_{k=\frac{m}2+1}^m (R_{i_k}^+(\lambda^2)L)\Big)
\end{equation*}
Consider the $\chi(x_n)$ term.  By~\eqref{eq:alldirected}, the first half of
the product carries an operator norm bound of $(C_n\, C_A)^{-\frac{m}{2}}
(r^{-1}C_nC_A)^{\frac{3R}{d(r)}}r^m$. The second half
contributes at most $(C_n C_A)^{m/2}$, based on~\eqref{eq:R_0} and
Lemma~\ref{lem:L}.  Put together, this product has an operator norm less than
$C_{n,A,r}\,r^m$, where $C_{n,A,r} = (r^{-1}C_nC_A)^{\frac{3R}{d(r)}}$.

The $\chi(-x_n)$ term has nearly identical
estimates, by duality. The adjoint of any directed resolvent
$R_{\Phi}^+(\lambda^2)$ is precisely
$R_{\tilde{\Phi}}^-(\lambda^2)$, with $\tilde{\Phi}$ being the
antipodal image of $\Phi$.  Because the order of multiplication is
reversed, one applies the geometric argument above (modulo the
antipodal map) to a product of the form
\begin{equation*}
\Big(\prod_{k=1}^{m/2} (L R_{i_k}^-(\lambda^2))\Big) \chi(-x_n),
\end{equation*}
which is an operator on $X_\lambda$.  The
estimates~\eqref{eq:distantL},~\eqref{eq:R_0}, and~\eqref{eq:R_d}
are used in the same manner as in deriving the main
bound~\eqref{eq:alldirected}.

According to Lemma~\ref{lem:combinatorics} there are at most $\delta^{1-n}
(C_n)^m$ directed products of length $m$.
To prove \eqref{eq:directedprod}, it therefore suffices to let
$r = \frac{1}{2C_n}$, and $\delta = \frac{1}{20m}$ so that the sum of the
operator norms of all
directed products is bounded by $20^{n-1} C_{n,A} m^{n-1}2^{-m}$.
This can be made smaller than $\frac{c}{2}$ by choosing $m$ sufficiently
large.
\end{proof}

As for the undirected products, recall that their defining feature
is the presence of adjacent resolvents $R_{i}^+(\lambda^2)$ oriented
in distinct directions.  The resulting oscillatory integral has no
region of stationary phase, and therefore exhibits improved bounds
at high energy provided the potential $A(x)$ is smooth.

\begin{lemma} \label{lem:nonstatphase}
Let $\Phi_1$ and $\Phi_2$ be chosen from a partition of unity of
$S^{n-1}$ so that their supports are separated by a distance greater
than $10\delta$. Suppose $A \in C^\infty(\R^n)$ with compact
support. Then for each $j \ge 0$, and any $N\ge1$,
\begin{equation}\label{eq:Phi1Phi2}
\big\| R_{d,\Phi_2}^+(\lambda^2) (L R_d^+(\lambda^2))^j L
R_{d,\Phi_1}^+(\lambda^2) \big\|_{X_\lambda \to X^*_\lambda} =
\O(\lambda^{-N})
\end{equation}
as $\lambda\to\infty$ and similarly for $R^-(\lambda^2)$.
\end{lemma}
\begin{proof}
In view of \eqref{eq:Rscale} and \eqref{eq:Hsplit} we can write
\begin{align*}
   &R_0^+(\lambda^2)(|x|)\eta_d(|x|)\Phi(x/|x|) \\&= \lambda^{n-2} R_0(1)(\lambda
   |x|) \eta_{d}(|x|)\Phi(x/|x|)\\
   &= \lambda^{\frac{n-3}{2}} \frac{e^{i\lambda|x|}}{|x|^{\frac{n-1}{2}}}\Big[
   a(\lambda|x|) +
   \frac{e^{-i\lambda|x|}b(\lambda|x|)}{|\lambda x|^{\frac{n-3}{2}}}\Big]\eta_{\lambda d}
   (\lambda |x|)\Phi(x/|x|)\\
   &= \lambda^{\frac{n-3}{2}}
   \frac{e^{i\lambda|x|}}{|x|^{\frac{n-1}{2}}} a_{\lambda d}(\lambda
   |x|) \Phi(x/|x|)
\end{align*}
where for arbitrary $\tilde d=\lambda d>0$
\[
a_{\tilde d}(r) = \Big[a(r) +
   \frac{e^{-ir}b(r)}{r^{\frac{n-3}{2}}}\Big]\eta_{\tilde d}(r)
\]
is supported on $\{r \ge d\}$ and satisfies the bounds, for all $\ell\ge0$
and $r>d$,
\[
|\partial_r^{\ell} a_{\tilde d}(r) |\le C_\ell\, \tilde
d^{-\frac{n-3}{2}}\, r^{-\ell} \le C_\ell\,
d^{-\frac{n-3}{2}}\, r^{-\ell}
\]
uniformly in $\lambda\ge1$. The kernel of the operator
of~\eqref{eq:Phi1Phi2} with $j=0$ equals
\begin{align}
K_{d,\lambda}(x,y)&:= \lambda^{n-3}\int_{\R^{n}} \frac{e^{i\lambda
|x-u|}}{|x-u|^{\frac{n-1}{2}}} \Phi_1\Big(\frac{x-u}{|x-u|}\Big)
a_{\lambda d}(\lambda|x-u|)(\nabla_u A(u)+\nonumber\\
&\qquad +A(u)\nabla_u)  \frac{e^{i\lambda |u-y|}}{
|u-y|^{\frac{n-1}{2}}}\, a_{\lambda d}(\lambda|u-y|)
\Phi_2\Big(\frac{u-y}{|u-y|}\Big)\, du \label{eq:Kdlam}
\end{align}
By our assumption on $A$ we can integrate by parts any number of
times in the $u$ variable since
\[
\big|\partial_u[ |x-u|+|u-y|]\big| = \Big|\frac{x-u}{|x-u|} -
\frac{u-y}{|u-y|} \Big| > \delta
\]
by the angular separation hypothesis between $\supp \Phi_1$ and
$\supp \Phi_2$. In conclusion, for arbitrary $N$,
\[
|K_{d,\lambda}(x,y)|\le C_N(A,d,\delta,n) \lambda^{-N}  \la
y\ra^{-\frac{n-1}{2}}\la x\ra^{-\frac{n-1}{2}}
\]
Here we also used the compact support assumption on $A$ which
restricts the size of $u$ in~\eqref{eq:Kdlam}. This kernel takes
$B\to B^*$ with norm $\les \lambda^{-N}$. In the same way one bounds
the kernel $D^\alpha_{x,y} K_{d,\lambda}(x,y)$ for any $\alpha$ which concludes the argument for $j=0$.

If
$j\ge1$, then write
\begin{align*}
   &R_0^+(\lambda^2)(|x|)[1-\eta_d(|x|)] \\
   &= \lambda^{n-2} R_0(1)(\lambda
   |x|) [1-\eta_{d}(|x|)]\\
   &= \lambda^{\frac{n-3}{2}} \frac{e^{i\lambda|x|}}{|x|^{\frac{n-1}{2}}}\Big[
   a(\lambda|x|) +
   \frac{e^{-i\lambda|x|}b(\lambda|x|)}{|\lambda x|^{\frac{n-3}{2}}}\Big][1-\eta_{\lambda d}(
   \lambda |x|)]\\
   &= \lambda^{\frac{n-3}{2}}
   \frac{e^{i\lambda|x|}}{|x|^{\frac{n-1}{2}}} b_{\lambda d}(\lambda
   |x|)
\end{align*}
where for arbitrary $\tilde d=\lambda d>0$
\[
b_{\tilde d}(r) = \Big[a(r) +
   \frac{e^{-ir}b(r)}{r^{\frac{n-3}{2}}}\Big][1-\eta_{\tilde d}(r)]
\]
satisfies the bounds
$|\partial_r^{\ell} b_{\tilde d}(r) |\le C_\ell\, \tilde
d^{\frac{n-3}{2}}\, r^{-\frac{n-3}{2}-\ell}$ for all $\ell \ge 0$ and $r>0$.
In particular,
\[
\big|D_x^{\ell}[b_{\lambda d}(\lambda|x|)]\big| \le C_\ell d^{\frac{n-3}{2}}
 |x|^{\frac{n-3}{2}- \ell}
\]
uniformly in $\lambda\ge1$. The kernel of the operator
of~\eqref{eq:Phi1Phi2} with $j>0$ now equals
\begin{align}
K_{j,d,\lambda}(x,y)&:= \lambda^{(n-3)(j+2)/2}\int_{\R^{(j+1)n}}
\frac{e^{i\lambda |x-u_0|}}{|x-u_0|^{\frac{n-1}{2}}}
\Phi_1\Big(\frac{x-u_0}{|x-u_0|}\Big) a_{\lambda
d}(\lambda|x-u_0|)\nonumber\\
& \prod_{i=0}^{j-1} (\nabla_{u_i} A(u_i) +A(u_i)\nabla_{u_i})
\frac{e^{i\lambda |u_i-u_{i+1}|}}{|u_i-u_{i+1}|^{\frac{n-1}{2}}}
b_{\lambda
d}(\lambda|u_i-u_{i+1}|) \nonumber\\
&(\nabla_{u_j} A(u_j)+A(u_j)\nabla_{u_j})\frac{e^{i\lambda |u_j-y|}}{
|u_j-y|^{\frac{n-1}{2}}}\, a_{\lambda d}(\lambda|u_j-y|)
\Phi_2\Big(\frac{u_j-y}{|u_j-y|}\Big)\, du \nonumber
\end{align}
We change variables
\[
w_i = \frac12 (u_i-u_{i+1}), \quad 0\le i\le j-1, \quad w_j =
\frac12 (u_0+u_j)
\]
so that $u_0= \sum_{i=0}^j w_i$, $u_j = w_j - \sum_{i=0}^{j-1} w_i$,
and $u_1=u_0-2w_0$, $u_2=u_1-2w_1$ etc. After this substitution we obtain
\begin{align}
K_{j,d,\lambda}(x,y)&:= c\lambda^{(n-3)(j+2)/2}\int_{\R^{(j+1)n}}
\frac{e^{i\lambda |x-u_0|}}{|x-u_0|^{\frac{n-1}{2}}}
\Phi_1\Big(\frac{x-u_0}{|x-u_0|}\Big) a_{\lambda
d}(\lambda|x-u_0|)\nonumber\\
& \prod_{i=0}^{j-1} (\nabla_{u_i} A(u_i) +A(u_i)\nabla_{u_i})
\frac{e^{2i\lambda |w_i|}}{|w_i|^{\frac{n-1}{2}}}
b_{\lambda
d}(2\lambda|w_i|) \nonumber\\
&(\nabla_{u_j} A(u_j)+A(u_j)\nabla_{u_j})\frac{e^{i\lambda |u_j-y|}}{
|u_j-y|^{\frac{n-1}{2}}}\, a_{\lambda d}(\lambda|u_j-y|)
\Phi_2\Big(\frac{u_j-y}{|u_j-y|}\Big)\, dw \nonumber
\end{align}
where it is understood that $u_0=u_0(w)$ and $u_j=u_j(w)$.  Of particular
interest, the phase functions involving $x,y$ contain the variable $w_j$, viz.
\[
\lambda |x-u_0| = \lambda \big|x-w_j- \sum_{i=0}^{j-1} w_i\big|,\qquad \lambda|y-u_j| = \lambda\big|y- w_j + \sum_{i=0}^{j-1} w_i\big|
\]
whereas none of the short range free resolvent kernels contains $w_j$. Thus, since
\[
 \Big|\frac{x-u_0}{|x-u_0|} -
\frac{u_j-y}{|u_j-y|} \Big| > \delta
\]
integration by parts in~$w_j$
yields as before
\[
|D_{x,y}^\alpha K_{j,d,\lambda}(x,y)|\le C_{N,\alpha}(A,d,\delta,n,j) \lambda^{-N}  \la
y\ra^{-\frac{n-1}{2}}\la x\ra^{-\frac{n-1}{2}}
\]
for any $\alpha, N$ and we are done.
\end{proof}
\begin{remark}
One should be careful that the integrand above is locally integrable.  Each
short range resolvent contains a singularity on the order of $|w_i|^{2-n}$
which becomes more severe with repeated differentiation.

Fortunately, in the full change of coordinates $\nabla_{u_i} =
\frac12(\nabla_{w_i} - \nabla_{w_{i-1}})$ for each $i = 1, 2, \ldots, j$,
while $\nabla_{u_0} = \frac12(\nabla_{w_0} + \nabla_{w_j})$.  Therefore
the dangerous piece $\frac{b_{\lambda d}(2\lambda|w_i|)}{|w_i|^{(n-1)/2}}$
experiences the $\nabla_{u_i}$ immediately preceding it, but no other
derivatives, creating local singularities no worse than $|w_i|^{1-n}$.
\end{remark}

Note that under the conditions of Lemma~\ref{lem:nonstatphase}
 each undirected product in~\eqref{eq:decomp2} satisfies the bound
\begin{equation} \label{eq:nonstatphase}
\bignorm[\prod_{k=1}^m \big(R_{i_k}(\lambda^2)L\big)][X^*_\lambda \to
X^*_\lambda] = \O(\lambda^{-N})
\end{equation}
for any $N\ge1$.
We now show by approximation that vanishing still holds for merely
continuous $A$, but without any control over the rate.

\begin{lemma} \label{lem:RiemannLebesgue}
Let $\Phi_1$ and $\Phi_2$ be chosen as in
Lemma~\ref{lem:nonstatphase}. Suppose $A$ is a continuous function
with $A\in Y$.  Then each undirected product in~\eqref{eq:decomp2}
satisfies the limiting bound
\begin{equation} \label{eq:RiemannLebesgue}
\lim_{\lambda \to \infty} \bignorm[ \prod_{k=1}^m
\big(R_{i_k}(\lambda^2)L\big)][X^*_\lambda \to X^*_\lambda] = \ 0.
\end{equation}
for any $\lambda\ge1$.
\end{lemma}
\begin{proof}
For any small $\gamma > 0$, there exists a smooth approximation
$A_\gamma \in C^\infty(\R^n)$ of compact support so that $\norm[A -
A_\gamma][Y] < \gamma$ and $\|A_\gamma\|_Y<2\|A\|_Y$. Define the
operator $L_\gamma$ accordingly. By Lemma~\ref{lem:L} and
Corollary~\ref{cor:R_0}
\begin{equation*}
\bignorm[\prod_{k=1}^m \big(R_{i_k}(\lambda^2)L\big) -
 \prod_{k=1}^m \big(R_{i_k}(\lambda^2)L_\gamma \big)][X^*_\lambda \to X^*_\lambda]
\les \gamma (2\|A\|_Y)^{m-1}
\end{equation*}
uniformly in $\lambda\ge1$.  Thus, by \eqref{eq:nonstatphase},
\begin{equation*}
\limsup_{\lambda \to \infty} \bignorm[\prod_{k=1}^m
\big(R_{i_k}(\lambda^2)L\big)][X^*_\lambda \to X^*_\lambda] \les \gamma
(2\|A\|_Y)^{m-1}
\end{equation*}
Sending $\gamma\to0$ finishes the proof.
\end{proof}

\begin{proof}[Proof of Lemma~\ref{lem:largem}]  Lemma~\ref{lem:directedprod} provides a
recipe for selecting a value of $m$, together with a partition of
unity $\{\Phi_i\}$ and a short-range threshold $d$, so that the sum
over all directed products in~\eqref{eq:decomp2} will be an operator
of norm less than $\frac{c}{2}$, or $C_{r}m^{n-1}r^m$.  We may choose $m$ so that
$2^m > C_{r}m^{n-1}$. This fixes the number of undirected
products as approximately $\delta^{m(1-n)} = (20 m)^{m(n-1)}$.  For
each of these, Lemma~\ref{lem:RiemannLebesgue} asserts that its
operator norm tends to zero as $\lambda \to \infty$.  The same is
true for the finite sum over all undirected products of length $m$.
In particular it is less than the directed product estimate provided $\lambda >
\lambda_1(m)$ is sufficiently large.
\end{proof}

\section{Small energies} \label{sec:small}

The remaining task is to verify that for sufficiently small $\lambda_0$ (and
following our convention regarding $\lambda^2\pm i0$ from before)
\begin{align}\label{eq:sm1}
&\sup_{0<\lambda<\lambda_0} \|Z_0 R_L(\lambda^2) Z_0^* \|_{2\to2} < \infty,
\end{align}
where $Z_0=\la x\ra^{-\sigma}|\nabla|^{\half}$ for some $\sigma > \frac12$.
As in the high energy case (and implicitly for intermediate energies), we need
to impose an invertibility condition which allows the resolvent
$R_L(\lambda^2)$ to be bounded between suitable spaces.
More precisely, by the resolvent identity,
\[ R_L(\lambda^2+i0)= (1+R_0(\lambda^2+i0)L)^{-1}R_0(\lambda^2+i0)
\]
provided the inverse on the right-hand side exists. We have
\begin{align*}
&\| Z_0 R_L(\lambda^2) Z_0^* \|_{2\to2} \\
&=  \| Z_0 (1+R_0(\lambda^2)L)^{-1} Z_0^{-1} Z_0 R_0(\lambda^2)
Z_0^*
 \|_{2\to2} \\
 &\le \| Z_0
(1+R_0(\lambda^2)L)^{-1} Z_0^{-1} \|_{2\to2}\,\,\, \|Z_0
R_0(\lambda^2) Z_0^*
 \|_{2\to2}
\end{align*}
By the smoothing properties of $Z_0$ relative to $H_0$,
\[
\sup_{\lambda}\|Z_0 R_0(\lambda^2) Z_0^*
 \|_{2\to2} <\infty
\]
Thus, it will suffice to verify that
\begin{equation}\label{eq:Z0small}
\sup_{|\lambda|<\lambda_0} \| Z_0 (1+R_0(\lambda^2)L)^{-1} Z_0^{-1}
\|_{2\to2}<\infty
\end{equation}
Let $G=R_0(0)$, and $B_\lambda=R_0(\lambda^2)-G$. We will prove that
under suitable conditions
\begin{align}
\label{eq:z0GL} &Z_0(I + GL)^{-1}Z_0^{-1} = (I+Z_0GLZ_0^{-1})^{-1}: L^2\to L^2\\
\label{eq:z0BL} &\|Z_0B_\lambda L Z_0^{-1}\|_{2\to2} \rightarrow
0\;\;\;\;\;\text{ as } \lambda \to 0
\end{align}
This implies \eqref{eq:Z0small} by summing the Neumann series directly.
The proof of \eqref{eq:z0GL} is a standard Fredholm alternative argument,
while~\eqref{eq:z0BL} will follow from properties of the kernel of $B_\lambda$.

\begin{lemma}\label{lem:frac}
Let $n>\beta > 0$, $\beta\ge\alpha$. Then
\[ \la \nabla\ra^{\alpha}|\nabla|^{-\beta} \::\: L^{2,\sigma_1}\to
L^{2,-\sigma_2}
\]
for all pairs $\sigma_1,\sigma_2> \beta-\frac{n}{2}$ satisfying
$\sigma_1+\sigma_2> \beta$, and is a compact operator whenever the
strict inequality $\beta > \alpha$ holds.

The same result holds when $\beta \le 0$ and $\beta \ge \alpha$, under the
conditions $\sigma_1, \sigma_2 > \beta - \frac{n}{2}$ and
$\sigma_1 + \sigma_2 \ge 0$.  Compactness in this case requires strict
inequalities for both $\beta > \alpha$ and $\sigma_1 + \sigma_2 > 0$.
\end{lemma}
\begin{proof}
Note that $\la x\ra^{-\eps}\la \nabla\ra^{-\delta}$ is compact on
$L^2$ for any choice of $\eps,\delta>0$. Therefore, it suffices to establish
the boundedness of $\japanese[x]^{-\sigma_2'}
\japanese[\nabla]^\beta |\nabla|^{-\beta}\japanese[x]^{-\sigma_1}$ on~$L^2$
for any value $\sigma_2' < \sigma_2$.
Strict inequalities are necessary to ensure that $\sigma_2'$ can also be
chosen to satisfy the hypotheses of the lemma.  Let
\[K(x,y)= \la x\ra^{-\sigma_2'}[\la\xi\ra^{\beta}|\xi|^{-\beta}]^{\vee}(x-y)
\la y\ra^{-\sigma_1}\]
Then
\[
|I - K(x,y)|\les \left\{ \begin{array}{cc} \la
x\ra^{-\sigma_2}|x-y|^{2-n}\la y\ra^{-\sigma_1} &
|x-y|\le 1\\
\la x\ra^{-\sigma_2}|x-y|^{\beta-n}\la y\ra^{-\sigma_1} & |x-y|>1
\end{array}\right.
\]
based on the fact that $1 - \japanese[\xi]^\beta|\xi|^{-\beta} \sim |\xi|^{-2}$
for all $|\xi| \ge 1$.

Define $K_{ij}(x,y)= [I- K(x,y)]\chi_i(x)\chi_j(y)$ where
$\chi_i(x)=1_{[|x|\le1]}$ if $i=0$ and $\chi_i(x)=1_{[2^{i-1}<|x|\le
2^i]}$ if $i\ge1$. Then
\[ |K_{ij}(x,y)| \les 2^{-i\sigma_2' - j\sigma_1} 2^{- \max(i,j)
(n-\beta)}
\]
provided $|i-j|>1$. Hence, $ K_1 := \sum_{|i-j|>1} K_{ij}$ defines a
bounded  operator (in fact, compact operator) on $L^2$ since its
Hilbert-Schmid norm is controlled by
\[
\| K_1\|_{HS}^2 \les \sum_{|i-j|>1} 2^{-2(i\sigma_2' + j\sigma_1)}
2^{- 2\max(i,j) (n-\beta)} 2^{(i+j)n} <\infty
\]
For fixed $i$, the sum over $j= i+2, i+3, \ldots $ is only finite if
$\sigma_1 > \beta-\frac{n}{2}$, and its value is then comparable
to $2^{-2i(\sigma_1 + \sigma_2' -\beta)}$.  Finite summation over $i$
then requires that $\sigma_1 + \sigma_2' > \beta$.
Similar conditions are noted, with the roles of $\sigma_1$ and $\sigma_2'$
reversed, when considering the summation over all $j > i+1$.

By almost orthogonality, $K_0=\sum_{|i-j|\le 1} K_{ij}$ satisfies
\[
\|K_0\|_{2\to2} \les \max_{|i-j|\le 1} \|K_{ij}\|_{2\to2}
\]
By Schur's test,
\[
\|K_{ij}\|_{2\to 2} \les 2^{-i(\sigma_1+\sigma_2')} 2^{i\max(\beta, 0)}
\]
when $|i -j| \le 1$, which is uniformly bounded provided $\sigma_1 + \sigma_2'
\ge \max(\beta, 0)$.  We are done evaluating the three components of
the decomposition $K = I - K_0 - K_1$.
\end{proof}
\begin{remark}
To be precise, the above proof did not capture the points $\beta = 0$,
$\sigma_1 + \sigma_2 = 0$, but this can be shown trivially as a special case.
\end{remark}

Next, we apply this result to prove compactness of the zero energy
operators.

\begin{lemma}
\label{lem:G} Assume that $L$ is as in \eqref{eq:ass1},
\eqref{eq:ass2}, and $Z_0 = \la x\ra^{-\half-}|\nabla|^{\half}$.
Then $Z_0 GL Z_0^{-1}$ is a compact operator on $L^2$.
\end{lemma}
\begin{proof} We shall use the decomposition
\[L= Y_1^* Z_1 + Y_2^* Z_2 + V
\] where
\[
  Y_1  := i \nabla |\nabla|^{-\frac12} \cdot Aw^{-1} , \quad Z_1
  := |
 \nabla|^{\frac12} w, \quad  Y_2 := Z_1,\quad Z_2 := Y_1
\]
As before, $w = \japanese[x]^{-\tau}$ for some
$\tau \in (\frac12, \frac12+\eps')$.  For convenience we will take
$\tau = \frac12(1+ \eps')$ in the calculations below.
Our goal is to prove that the operators
\[\begin{split}
O_1&= \la x\ra^{-\sigma}|\nabla|^{\half}GV|\nabla|^{-\half}\la
x\ra^{\sigma} \\
O_2&= \la x\ra^{-\sigma}|\nabla|^{\half}G Y_1^* Z_1|\nabla|^{-\half}\la x\ra^{\sigma}\\
O_3&= \la x\ra^{-\sigma}|\nabla|^{\half}G Y_2^* Z_2|\nabla|^{-\half}\la x\ra^{\sigma}
\end{split}
\]
are compact in $L^2$ for some $\sigma > \frac12$.
In what follows, we will use the commutator
bounds of the appendix without further mention. The same applies to
the fact that $|\nabla|^{\frac12} A\la x\ra^{1+}
|\nabla|^{-\frac12}$
(and therefore its adjoint) are bounded on $L^2$,
see Lemma~\ref{lem:Amap} above.

For $O_1$, it suffices to observe that
\[
\la x\ra^{-\sigma}|\nabla|^{\half}GV|\nabla|^{-\half} \la x\ra^{\sigma}=
\big(\la x\ra^{-\sigma}|\nabla|^{-\frac32}
\la x\ra^{-1}\big)\big(\la x\ra V |\nabla|^{-\half}\la x\ra^{\sigma}\big)
\]
is compact by Lemma~\ref{lem:frac} provided
$\sigma \in (\frac12, \frac12 + \eps)$.

Denote the bounded and compact operators on $L^2$ by $\calB$ and
$\calC$, respectively. Then,
\begin{align*}
  O_2 &= \la x\ra^{-\sigma}|\nabla|^{-\frac32} Aw^{-1} \nabla |\nabla|^{-\half}
  \big(|\nabla|^{\frac12}w |\nabla|^{-\frac12}\la
  x\ra^{\sigma}\big)\\
  &\in  \la x\ra^{-\sigma}|\nabla|^{-\frac32}w |\nabla|^{\frac12}
  \big(|\nabla|^{-\frac12} Aw^{-2} |\nabla|^{\half}\big) \nabla
  |\nabla|^{-1} \calB \\
  &\subset  \big(\la x\ra^{-\sigma}|\nabla|^{-1}  \la x\ra^{-\frac12}\big)
  \big(\la x\ra^{\frac12}|\nabla|^{-\frac12} w |\nabla|^{\frac12}\big)
\calB \subset \calC.
\end{align*}
requires $\sigma \in (\frac12, \tau)$.  Finally, under the same conditions,
\begin{align*}
  O_3 &= \la x\ra^{-\sigma}|\nabla|^{-\frac32}
  w|\nabla|^{\half} (\nabla |\nabla|^{-1})
 \big(|\nabla|^{\half} Aw^{-2} |\nabla|^{-\half}\big)
  \big(|\nabla|^{\half} w |\nabla|^{-\half} \la x\ra^{\sigma}\big) \\
  &\in \la x\ra^{-\sigma}|\nabla|^{-\frac32}
  w|\nabla|^{\half} \calB \\&
  \subset \big(\la x\ra^{-\sigma}|\nabla|^{-1}\la x\ra^{-\half}\big)
  \big(\la x\ra^{\half} |\nabla|^{-\half} w|\nabla|^{\half}\big) \calB
  \subset \calC
\end{align*}
and we are done.
\end{proof}

As an immediate consequence we arrive at the following.
\begin{cor}
  \label{cor:invert}
Let $Z_0=\la x\ra^{-\sigma}|\nabla|^{\half}$
with $\sigma \in (\frac12, \frac12+\eps')$.
Assume that $\ker(I+Z_0 GLZ_0^{-1})=\{0\}$ as an operator on $L^2(\R^n)$.
Then $I+ Z_0 GL Z_0^{-1}$ is invertible on $L^2$.
\end{cor}
\begin{proof}
The statement follows from  Fredholm's alternative. Note that
\[(I+Z_0 GL Z_0^{-1} )^{-1} = Z_0 (I+GL)^{-1} Z_0^{-1}
\]
where $GL$ on the right-hand side is an operator on
$Z_0^{-1}(L^2(\R^n))$.
\end{proof}

Now, we verify the vanishing norm condition \eqref{eq:z0BL}.

\begin{lemma}For any $L = i(A \cdot \nabla + \nabla \cdot A) + V$ satisfying
conditions \eqref{eq:ass1},~\eqref{eq:ass2}, one has
\[
\lim_{\lambda\to 0^+} \|\la x\ra^{-\sigma}|\nabla|^{\half}B_\lambda
L|\nabla|^{-\half} \la x\ra^{\sigma}\|_{2\to 2} =0\\
\]
\end{lemma}
\begin{proof} By the commutator identities, the assumptions on $A$ and $V$, and fractional integration, the
claim above is a consequence of the bound
\begin{equation}
\label{eq:sample} \lim_{\lambda\to 0^+} \|\la x\ra^{-\sigma} \nabla
B_\lambda \la x\ra^{-\sigma}\|_{2\to 2} =0
\end{equation}
To be precise, the reduction proceeds as follows.  Recall that $B_\lambda$
is defined as a function of $-\Delta$, and therefore commutes with all
derivatives.  First,
\begin{align*}
\la x\ra^{-\sigma} |\nabla|^{\half}B_\lambda V|\nabla|^{-\half}
\la x \ra^{\sigma} &=
\sum_{i=1}^n \big(\la x\ra^{-\sigma}\partial_i B_\lambda
\la x\ra^{-\sigma}\big) \big(\la x\ra^{\sigma}\partial_i|\nabla|^{-\frac32}
V|\nabla|^{-\half}\la x\ra^{\sigma}\big) \\
&= \sum_i \big(\la x\ra^{-\sigma}\partial_i B_\lambda \la x\ra^{-\sigma}\big)
S_1^i
\end{align*}
where each $S_1^i$ is bounded on $L^2$ by the fractional integration
estimates in Lemma~\ref{lem:frac}.  For the gradient term $\nabla \cdot A$
we have
\[
\la x\ra^{-\sigma}|\nabla|^{-\half}B_\lambda \nabla \cdot A |\nabla|^{-\half}
\la x\ra^{\sigma} =
\big(\la x\ra^{-\sigma} \nabla B_\lambda \la x\ra^{-\sigma} \big) \cdot S_2
\]
where the operator
\[
S_2 = \big(\la x\ra^{\sigma}|\nabla|^{\half} w |\nabla|^{-\half}\big)
      \big(|\nabla|^{\half} Aw^{-2} |\nabla|^{-\half}\big)
      \big(|\nabla|^{\half} w |\nabla|^{-\half}\la x\ra^{\sigma}\big)
\]
is bounded on $L^2$ by Lemma~\ref{lem:Amap} and the commutator estimates in
Lemma~\ref{lem:comm}.  The second gradient term, $A \cdot \nabla$,
requires a slightly more intricate decomposition.
\[
\la x\ra^{-\sigma}|\nabla|^{-\half}B_\lambda A \cdot \nabla |\nabla|^{-\half}
\la x\ra^{\sigma} = \sum_{i,j = 1}^n
\big(\la x\ra^{-\sigma} \partial_i B_\lambda \la x\ra^{-\sigma} \big)
 S_3^{i,j}
\]
where each $S_3^{i,j}$ has the structure
\begin{align*}
S_{i,j} &=
\big(\la x\ra^{\sigma} \partial_i|\nabla|^{-1} \la x\ra^{-\sigma} \big)
\big(\la x\ra^{\sigma}|\nabla|^{-\half} w |\nabla|^{\half}\big) \\
&\quad \times \big(|\nabla|^{-\half} A_j w^{-2} |\nabla|^{\half}\big)
\big(|\nabla|^{-\half} w |\nabla|^{\half} \la x\ra^{\sigma}\big)
\big(\la x\ra^{-\sigma} \partial_j |\nabla|^{-1} \la x\ra^{\sigma}\big)
\end{align*}
The central term is bounded on $L^2$ by Lemma~\ref{lem:Amap}; it is flanked
by a pair of commutators as in Lemma~\ref{lem:comm}.  The boundedness of
the outer operators simply reflects the boundedness of the Riesz transforms
on the weighted space $\la x\ra^{\pm \sigma}L^2$.

Now it remains to verify \eqref{eq:sample}.
With the notation of Section~\ref{sec:directed}, we have
\[
B_\lambda(x,y)=
\frac{b(\lambda|x-y|)}{|x-y|^{n-2}}-\frac{b(0)}{|x-y|^{n-2}}+\lambda^{\frac{n-3}{2}}
e^{i\lambda|x-y|}\frac{a(\lambda|x-y|)}{|x-y|^{\frac{n-1}{2}}}
\]
We write  $\nabla B_\lambda =T_{\lambda,0}+T_{\lambda,1}$, where
\begin{align*}
|T_{\lambda,0}(x,y)|& \les \frac{\lambda}{|x-y|^{n-2}} \\
 T_{\lambda,1}(x,y)  & =\lambda^{\frac{n-3}{2}}\Big(\lambda+\frac{1}{|x-y|}\Big)
 e^{i\lambda|x-y|}\frac{\tilde a(\lambda|x-y|)}{|x-y|^{\frac{n-1}{2}}}
\end{align*}
where $\tilde a $ is a modified symbol with the same properties as
$a$.  The support of $T_{\lambda,0}$ is restricted to the set
$\{\lambda|x-y| \les 1\}$, so the point-wise estimate
\[
|T_{\lambda,0}(x,y)| \les
\frac{\lambda^{\sigma-\half}}{|x-y|^{n-\half-\sigma}}
\]
is also valid.  Thanks to the positive power of $\lambda$ in the numerator,
$T_{\lambda,0}$ satisfies \eqref{eq:sample} by fractional integration.
$T_{\lambda,1}$ requires more care.
Let $\chi$ be a smooth cut-off for the
region $\{x:|x|\sim 1\}$. It suffices to prove that for $R_2\gtrsim
R_1 > 1 $ and  $R_2\gtrsim 1/\lambda$ (since $\tilde a(\lambda|x-y|)=0$
for $|x-y|<1/\lambda$),
\begin{equation*}
\label{eq:tl1}\|\chi(x/R_1) T_{\lambda,1}(x,y) \chi(y/R_2)\|_{2\to2}
\les \lambda^{\eps}R_1^{\half}R_2^{\half+\eps}
\end{equation*}
This, however, is an almost immediate corollary of Lemma~\ref{lem:VCP}.
We can chop $T_{\lambda,1}$ into finitely many conical pieces with
$\delta\sim 1$.  A properly scaled version of~\eqref{eq:Tbound} states that
\[
\|\chi(x/R_1) T_{\lambda,1}(x,y) \chi(y/R_2)\|_{2\to2}
\les \sqrt{R_1 R_2} \le \lambda^{\eps} R_1^{\half} R_2^{\half+\eps}
\]
for each piece, because $\lambda R_2 > 1$.
\end{proof}

We now relate the condition in Corollary~\ref{cor:invert} to the
notion of resonance and/or eigenvalue at zero.

\begin{lemma}
  \label{lem:reson}
  Suppose that zero is neither an eigenvalue nor a resonance of
  $H$. Then
  \[
\ker(I+Z_0 GL Z_0^{-1})=\{0\} \text{\ \ on\ \ }L^2(\R^n)
  \]
  for $Z_0=\la x\ra^{-\sigma}|\nabla|^{\half}$, with
$\sigma \in (\half, \half+\eps')$
In particular, \eqref{eq:sm1}
holds for sufficiently small~$\lambda_0$.
\end{lemma}
\begin{proof}
Suppose $f\in L^2(\R^n)$ satisfies
\[
f+ Z_0 GL Z_0^{-1} f =0
\]
We proved in Lemma~\ref{lem:G} that $Z_0GLZ_0^{-1}:L^2\to L^2$. By a
simple modification of the proof, we can obtain
$Z_0GLZ_0^{-1}:L^{2,\rho}\to L^{2,\rho+\eps}$ for $\rho\in[0,\frac{n}{2}-1)$
and for some fixed $\eps>0$ which depends on the decay rates of $A$ and
$V$.  This is done by commuting $\japanese[x]^\rho$ through each of the
expressions $O_1$, $O_2$, $O_3$.  Two representative examples from
the study of $O_2$ are presented below.
\begin{multline*}
\la x\ra^{\rho}\big(\la x\ra^{\half}|\nabla|^{-\half} w |\nabla|^{\half}\big)
\la x\ra^{-\rho} \\
=
\big(\la x\ra^{\rho+\half}|\nabla|^{-\half}\la x\ra^{-\rho+\eps}w|\nabla|^{\half}\big)
\big(|\nabla|^{-\half}\la x\ra^{-\rho-\eps}|\nabla|^{\half}\la x\ra^{-\rho}\big)
\end{multline*}
The constraints in Lemma~\ref{lem:comm} require that $\rho+\half <
\frac{n-1}{2}$, hence the upper bound $\rho < \frac{n}{2} -1$. The
second example involves the non-smooth function $A$, however it does
not pose any difficulties.
\begin{multline*}
\la x\ra^{\rho}\big(|\nabla|^{\half} Aw^{-2}|\nabla|^{-\half}\big)
\la x\ra^{-\rho} = \\
\big(\la x\ra^{\rho}|\nabla|^{\half}\la x\ra^{-\rho-\eps}|\nabla|^{-\half}\big)
\big(|\nabla|^{\half}A w^{-2}\la x\ra^{2\eps}|\nabla|^{-\half}\big)
\big(|\nabla|^{\half}\la x\ra^{\rho-\eps} |\nabla|^{-\half}
\la x\ra^{-\rho}\big)
\end{multline*}
These examples also demonstrate the flexibility to adjust the
weights up or down by a factor of $\japanese[x]^{\eps}$.  In this
manner it is possible to accommodate a weight of
$\japanese[x]^{\rho+\eps}$ on one side and $\japanese[x]^{-\rho}$ on
the other.

By iterating the relation $f = -Z_0GLZ_0^{-1}f$ a sufficient number of
times, it follows that $f\in L^{2,(n-2)/2}$. Set $h:= Z_0^{-1} f$. Then
$h=-GL h$. We have $h\in \cap_{\tau>\frac{n-4}2} L^{2,\tau}(\R^n)$
since $Z_0^{-1}:L^{2,\rho}\to L^{2,\rho-1-}$.
It follows, see \cite{JenKat}, that $Hh=0$ in the
distributional sense.  If $n \ge 5$, we see that $h$ is a true $L^2$
eigenfunction of $H$.  In dimensions $n=3,4$ we can only conclude that
$h$ exists in polynomially weighted $L^2$, making it indicative of a resonance.
However, by our assumption on zero energy it
follows that $h=0$ and therefore $f=0$ as desired.
\end{proof}

\section{Appendix: H\"ormander's  Plancherel theorem, a commutator bound, and the fractional Leibniz rule.}

The following is a version of H\"ormander's variable coefficient Plancherel theorem, see Theorem~1.1 in~\cite{Hor2}.

\begin{prop}\label{prop:vcp}
Let $a=a(u,v), \Psi=\Psi(u,v) \in C^\infty(\R^{n}\times \R^{m})$ with $\supp (a)\subset B_n(0,1)\times B_m(0,1)$
and $\Psi$ real-valued. Write $u=(u',u'')$, $v=(v',v'')$ and assume that $u', v'\in \R^{n_1}$ where $1\le n_1\le
\min(n,m)$. Assume that on the support of $a$, for some finite constants $\mu>0$, and $M>1$,
\begin{align*}
|\nabla_{u'} \Psi(u',u'', v',v'') - \nabla_{u'} \Psi(u',u'',
w',w'')| &\ge \mu|v'-w'|\\
\sup_{|\alpha|\le n_1+1}|D_{u'}^\alpha[\nabla_{u'} \Psi(u',u'', v',v'') - \nabla_{u'}
\Psi(u',u'', w',w'')]| &\le M\mu|v'-w'|\\
\sup_{|\alpha|\le n_1+1} \|\partial_{u'}^\alpha\, a\|_\infty &\le M
\end{align*}
Then the operator
\[
(T_\lambda f)(u) := \int e^{i\lambda\Psi(u,v)} a(u,v) f(v)\, dv
\]
satisfies the estimate
\[
\|T_\lambda f\|_{L^2(\R^m)} \le C(n,m,M) \la\lambda\mu\ra^{-\frac{n_1}{2}} \|f\|_{L^2(\R^n)}
\]
for all $\lambda>0$. The constant $C$ depends only on the dimensions $n,m$ and $M$.
\end{prop}
\begin{proof}
Define
\[ T^{(u'',v'')}_\lambda f(u') = \int_{\R^{n_1}} e^{i\lambda\Psi(u',u'',v',v'')} a(u',u'',v',v'') f(v')\, dv' \]
where $(u'',v'')$ are fixed parameters. We have
\[\begin{split}
&((T^{(u'',v'')}_\lambda)^* T^{(u'',v'')}_\lambda f)(w') = \int K^{(u'',v'')}(w',v')f(v')\, dv'\\
&K^{(u'',v'')}(w',v')  = \int e^{i\lambda[\Psi(u,v)-\Psi(u,w)]}\; \bar{a}(u,w) a(u,v)\, du'
\end{split}\]
Introduce the differential operator
\[
L = -i\lambda^{-1}\frac{\nabla_{u'} \Psi(u',u'', v',v'') - \nabla_{u'} \Psi(u',u'', w',w'')}{|\nabla_{u'}
\Psi(u',u'', v',v'') - \nabla_{u'} \Psi(u',u'', w',w'')|^2} \cdot\nabla_{u'}
\]
Note that for any $|\beta|\le n_1+1$,
\[\begin{split}
&L e^{i\lambda[\Psi(u',u'',v',v'')-\Psi(u',u'',w',w'')]} =
e^{i\lambda[\Psi(u',u'',v',v'')-\Psi(u',u'',w',w'')]}\\
&\Big|D_{u'}^\beta \Big[\frac{\nabla_{u'}\Psi(u',u'', v',v'') - \nabla_{u'} \Psi(u',u'', w',w'')}{|\nabla_{u'}
\Psi(u',u'', v',v'') - \nabla_{u'} \Psi(u',u'', w',w'')|^2}\Big] \Big| \les (\mu |v'-w'|)^{-1}
\end{split}
\]
Hence, for any $N$,
\[\begin{split}
&K^{(u'',v'')}(w',v') = \int e^{i\lambda[\Psi(u,v)-\Psi(u,w)]}\; (L^*)^N\big[\bar{a}(u,w) a(u,v)\big]\, du'
\end{split}
\]
so that by our assumptions,
\[
|K^{(u'',v'')}(w',v')|\le C(n,m,M) \la \lambda\mu| v'-w'|\ra^{-n_1-1}
\]
The lemma now follows by Schur's test. In fact there is the stronger estimate
\[
\|T_\lambda f\|_{L^\infty_{u''}L^2_{u'}} \le C(n,m,M) \la\mu\lambda\ra^{-\frac{n_1}{2}}
\|f\|_{L^1_{v''}L^2_{v'}}
\]
for all $\lambda>0$.
\end{proof}

Next, we present three commutator bounds.
The first one is from H\"ormander~\cite{Hor}, and the
second two are variants which are most likely standard.

\begin{lemma}
  \label{lem:comm} Suppose $\sigma,\tau \in\R$. Then
\[
\la \nabla\ra^\tau w_\sigma^{-1} \la \nabla\ra^{-\tau} w_\sigma
\]
is $L^2$ bounded on $\R^n$.  Further, let  $\sigma_1>\sigma_2$ with
$\sigma_1 > -n$ and $\frac{n-1}{2} > \sigma_2$. Then
\[
|\nabla|^{\frac12} w_{\sigma_1} |\nabla|^{-\frac12} w_{\sigma_2}^{-1}
\]
is also $L^2$ bounded on $\R^n$.  The reversed commutator
\[
|\nabla|^{-\frac12} w_{\sigma_1} |\nabla|^{\frac12} w_{\sigma_2}^{-1}
\]
is $L^2$ bounded on $\R^n$, $n>1$ provided $\sigma_1 > \sigma_2$ with
$\sigma_1 > -n$ and  $\frac{n+1}{2}>\sigma_2$.
In all these expressions, $w_\sigma(x) := \la x\ra^{-\sigma}$.
\end{lemma}
\begin{proof}
The first statement is from \cite{Hor}, see Definition~30.2.2, as
well as Theorem~18.1.13. For the second we write, with
$1=\chi_{[|\xi|>1]} + \chi_{[|\xi|\le 1]}$ a smooth partition of
unity,
\begin{align*}
& |\nabla|^{\frac12} w_{\sigma_1} |\nabla|^{-\frac12}
  w_{\sigma_2}^{-1} \\&= |\nabla|^{\frac12}\chi_{[|\nabla|>1]} w_{\sigma_1}
  |\nabla|^{-\frac12} \chi_{[|\nabla|>1]}
  w_{\sigma_2}^{-1} + |\nabla|^{\frac12}\chi_{[|\nabla|>1]} w_{\sigma_1}
  |\nabla|^{-\frac12} \chi_{[|\nabla|<1]}
  w_{\sigma_2}^{-1} \\
& + |\nabla|^{\frac12}\chi_{[|\nabla|<1]} w_{\sigma_1}
  |\nabla|^{-\frac12} \chi_{[|\nabla|>1]}
  w_{\sigma_2}^{-1} + |\nabla|^{\frac12}\chi_{[|\nabla|<1]} w_{\sigma_1}
  |\nabla|^{-\frac12} \chi_{[|\nabla|<1]}
  w_{\sigma_2}^{-1}
\end{align*}
We denote the terms on the right-hand side, in this order, as
high-high, high-low, low-high, and low-low, respectively.
By the first commutator bound it will suffice to deal with the low-low
and high-low cases. We shall do this
 by means of standard Littlewood-Paley
projections $P_j f = \phi_j \ast f$ where $P_j$ denotes a projection
onto frequencies $2^j$.
We start with the low-low case. With $w_j=w_{\sigma_j}$ it is of the
form
\begin{align*}
 \sum_{j,k\ge0} 2^{-\frac{j}{2}+\frac{k}{2}} P_{-j} \big( w_1
P_{-k} (w_2^{-1}  f)\big)(x) &= \sum_{j,k\ge0}
2^{-\frac{j}{2}+\frac{k}{2}}
\phi_{-j} \ast \big( w_1 [\phi_{-k} \ast \frac{f}{w_2}] )(x) \\
& =\sum_{j,k\ge0} 2^{-\frac{j}{2}+\frac{k}{2}} \phi_{-j} \ast \big(
w_1 [\phi_{-k} \ast \frac{f}{w_2}] )(x) \\
&= \sum_{j,k\ge0} 2^{-\frac{j}{2}+\frac{k}{2}} \int \chi_{[|v|\sim
2^{-j}]} G_{k,f}(v)
e^{-ivx} \, dv \\ {\rm where} \quad
G_{k,f}(v) &= \int w_1(y) \phi_{-k}(y-z) e^{ivy}\, dy
\frac{f(z)}{w_2(z)}\, dz \\
&= \int P_{-k} (w_1 e^{iv\cdot}) (z)\frac{f(z)}{w_2(z)}\, dz
\end{align*}
Since $|\widehat{w}_1(\xi)|\les |\xi|^{-(n-\sigma_1)}$, it follows
that, provided $|j-k|\gg1$,
\begin{align*} \sup_{|v|\sim 2^{-j}} \| P_{-k} (w_1
e^{iv\cdot}) \|_{L^{2,\sigma_2}} &\les 2^{k\sigma_2}
2^{-\frac{kn}{2}}
(2^{-j}+2^{-k})^{-(n-\sigma_1)} \\
\sup_{|v|\sim 2^{-j}} |G_{k,f}(v)| &\les 2^{k\sigma_2}
2^{-\frac{kn}{2}} (2^{-j}+2^{-k})^{-(n-\sigma_1)} \|f\|_2 \\
\|P_{-j} \big( w_1 P_{-k} (w_2^{-1}  f)\big)\|_2 &\les
2^{-j\frac{n}{2}}2^{k\sigma_2} 2^{-\frac{kn}{2}}
(2^{-j}+2^{-k})^{-(n-\sigma_1)} \|f\|_2,
\end{align*}
whereas by Schur's lemma, for the case $|j-k|\les1$,
\[
\|P_{-j} \big( w_1 P_{-k} (w_2^{-1}  f)\big)\|_2 \les
2^{-j(\min(\sigma_1, n) - \max(\sigma_2, -n))}
\]
 In conclusion,
\begin{align*}
\sum_{j,k\ge0} &2^{-\frac{j}{2}+\frac{k}{2}} \|P_{-j} \big( w_1
P_{-k} (w_2^{-1}  f)\big)\|_2  \\
&\quad \les \quad \sum_{j\ge k\ge0} 2^{-(j-k)(n+1)/2}
2^{-k(\min(\sigma_1,n)-\max(\sigma_2,-n))} \|f\|_2 \\
&\qquad + \sum_{k>j\ge0} 2^{-(k-j)((n-1)/2-\sigma_2)}
2^{-j(\min(\sigma_1,n)-\max(\sigma_2,-n))} \|f\|_2 \\
&\quad \les \|f\|_2
\end{align*}

Next, consider the high-low case. It takes the form
\[
\sum_{j,k\ge0} 2^{\frac{j}{2}+\frac{k}{2}} P_{j} \big( w_1 P_{-k}
(w_2^{-1}  f)\big)(x) = \sum_{j,k\ge0} 2^{\frac{j}{2}+\frac{k}{2}} \int
\chi_{[|v|\sim 2^{j}]} G_{k,f}(v) e^{-ivx} \, dv
\]
with $G_{k,v}$ as above.  To be precise, the marginal
cases $|j| + |k| \le 1$ are already part of the previous estimate.
Since $|\widehat{w}_1(\xi)|\les |\xi|^{-N}$ for $|\xi|>1$,
\begin{align*}
\sup_{[|v|\sim 2^j]} \| P_{-k} (w_1 e^{iv\cdot})\|_{L^{2,\sigma_2}}
&\les 2^{-k\frac{n}{2}} 2^{-jN} 2^{k\sigma_2} \\
\sup_{[|v|\sim 2^j]} |G_{k,f}(v)| &\les 2^{-k\frac{n}{2}} 2^{-jN}
2^{k\sigma_2} \|f\|_2 \\
\|P_{-j} \big( w_1 P_{-k} (w_2^{-1}  f)\big)\|_2 &\les
2^{(j-k)\frac{n}{2} -jN + k\sigma_2} \|f\|_2
\end{align*}
and thus, finally,
\[\begin{split}
&\sum_{j,k\ge0} 2^{\frac{j}{2}+\frac{k}{2}} \|P_{j} \big( w_1 P_{-k}
(w_2^{-1}  f)\big)\|_2 \\
&\les  \sum_{j,k\ge0} 2^{-j(N-(n+1)/2)} 2^{-k((n-1)/2 - \sigma_2)}
\|f\|_2\les \|f\|_2
\end{split}
\]
and we are done with the second statement.  The third statement is
verified using the same Littlewood-Paley decomposition and many of
the same estimates. The high-high term is again dominated by the
corresponding piece of the first commutator bound.  The low-low and
high-low terms follow the analysis above since they are concerned
with the same operators $P_{\pm j}\big(w_1P_{-k}(w_2^{-1}f)\big)$.
Thus we can quickly sum
\begin{align*}
\sum_{j,k\ge0} &2^{\frac{j}{2}-\frac{k}{2}} \|P_{-j} \big( w_1
P_{-k} (w_2^{-1}  f)\big)\|_2  \\
&\quad \les \quad \sum_{j\ge k\ge0} 2^{-(j-k)(n-1)/2}
2^{-k(\min(\sigma_1,n)-\max(\sigma_2,-n))} \|f\|_2 \\
&\qquad + \sum_{k>j\ge0} 2^{-(k-j)((n+1)/2-\sigma_2)}
2^{-j(\min(\sigma_1,n)-\max(\sigma_2,-n))} \|f\|_2 \\
&\quad \les \|f\|_2
\end{align*}
for the low-low term, and
\[\begin{split}
&\sum_{j,k\ge0} 2^{-\frac{j}{2}-\frac{k}{2}} \|P_{j} \big( w_1 P_{-k}
(w_2^{-1}  f)\big)\|_2 \\
&\les  \sum_{j,k\ge0} 2^{-j(N-(n-1)/2)} 2^{-k((n+1)/2 - \sigma_2)}
\|f\|_2\les \|f\|_2
\end{split}
\]
for the high-low term. Finally, the low-high term is also a concern.
Similar to the high-low case it takes the form
\[
\sum_{j,k\ge0} 2^{\frac{j}{2}+\frac{k}{2}} P_{-j} \big( w_1 P_{k}
(w_2^{-1}  f)\big)(x) = \sum_{j,k\ge0} 2^{\frac{j}{2}+\frac{k}{2}} \int
\chi_{[|v|\sim 2^{-j}]} G_{-k,f}(v) e^{-ivx} \, dv
\]
where $G_{-k,f}$ is the inverse Fourier transform of $w_1[\phi_{-k}\ast
(w_2^{-1}f)]$ as before.  Using the fact that $\widehat{w}_1(\xi)$ decays
rapidly when $|\xi| \ge 1$, we can conclude that
\begin{align*}
\sup_{[|v|\sim 2^{-j}]} \| P_{k} (w_1 e^{iv\cdot})\|_{L^{2,\sigma_2}}
&\les 2^{k\frac{n}{2}} 2^{-kN} \\
\sup_{[|v|\sim 2^{-j}]} |G_{-k,f}(v)| &\les 2^{k\frac{n}{2}} 2^{-kN}
\|f\|_2 \\
\|P_{-j} \big( w_1 P_{k} (w_2^{-1}  f)\big)\|_2 &\les
2^{(k-j)\frac{n}{2} -kN} \|f\|_2
\end{align*}
leading to the summation
\[\begin{split}
&\sum_{j,k\ge0} 2^{\frac{j}{2}+\frac{k}{2}} \|P_{-j} \big( w_1 P_{k}
(w_2^{-1}  f)\big)\|_2 \\
&\les  \sum_{j,k\ge0} 2^{-j((n-1)/2)}2^{-k(N - (n+1)/2)}
\|f\|_2\les \|f\|_2
\end{split}
\]
\end{proof}

Finally, we state a fractional Leibniz rule which is used in the proof of
Lemma~\ref{lem:Amap}

\begin{lemma}\label{lem:Leib}
For any $\alpha\ge0$, $1<p<\infty$, and arbitrarily small
$\gamma>0$,
\[
\| \,|\nabla|^\alpha (fg)\|_p %&
\le C_1 \Big[ \| \, |\nabla|^\alpha
f\|_{p_1} \|g\|_{q_1} + \| \, |\nabla|^\alpha
g\|_{p_2} \|f\|_{q_2} \Big]
\]
provided $\frac{1}{p} = \frac{1}{p_1}+\frac{1}{q_1} =
\frac{1}{p_2}+\frac{1}{q_2}$,
$p\le p_1,p_2 < \infty$, $p< q_1,q_2 \le\infty$. The constant $C_1$
depends on $n, \alpha,p,p_1,p_2,q_1,q_2$.
\end{lemma}
\begin{proof} This is standard para-differential calculus. See
for example~\cite{Tay}, page 105.
\end{proof}


\begin{thebibliography}{99}

\bibitem{AbSteg} Abramowitz, M., Stegun, I.\ A. {\em Handbook of mathematical
functions with formulas, graphs, and mathematical tables.} National
Bureau of Standards Applied Mathematics Series, 55. For sale by the
Superintendent of Documents, U.S. Government Printing Office,
Washington, D.C. 1964

\bibitem{Agmon} Agmon, S. {\em Spectral properties of Schr\"odinger
operators and scattering theory.} Ann.\ Scuola Norm.\ Sup.\ Pisa
Cl.\ Sci. (4) 2 (1975), no.~2, 151--218.

%\bibitem{AvHeSim} Avron, J.\ E., Herbst, I.\ W., Simon, B.
%{\em Schr\"odinger operators with magnetic fields. III. Atoms in
%homogeneous magnetic field.}  Comm.\ Math.\ Phys.~79  (1981), no.~4,
%529--572.

\bibitem{CK} Christ, M., Kiselev, A. {\em Maximal functions associated with
filtrations,} J.\ Funct.\ Anal.\ 179 (2001), 409-425.

\bibitem{CS1} Constantin, P., Saut, J.-C.
{\em Local smoothing properties of dispersive equations.} J.\ Amer.\
Math.\ Soc.~1 (1988), no.~2, 413--439.

\bibitem{CS2} Constantin, P., Saut, J.-C.
{\em Local smoothing properties of Schr\"odinger equations.} Indiana
Univ.\ Math.\ J.~38 (1989), no.~3, 791--810.

%\bibitem{CFKS} Cycon, H. L., Froese, R. G., Kirsch, W., Simon, B.
%{\em Schr\"odinger operators with application to quantum mechanics
%and global geometry.}
% Texts and Monographs in Physics. Springer Study Edition. Springer-Verlag, Berlin, 1987.

\bibitem{EGS} Erdogan, M., Goldberg, M., Schlag, W. {\em Strichartz and smoothing estimates for Schr\"odinger operators with
large magnetic potentials in~$\R^3$}, to appear, JEMS.

\bibitem{Erdos} Erd\"os, L. {\em Recent developments in quantum mechanics with magnetic fields},
"Spectral Theory and Mathematical Physics",  Proceedings of Symposia
in Pure Mathematics, vol.~76, AMS, Providence, 2007.

\bibitem{GST} Georgiev, V., Stefanov, A., Tarulli, M. {\em Smoothing
-  Strichartz estimates for the Schr\"odinger equation with small
magnetic potential}, Discrete Contin.\ Dyn.\ Syst.~17 (2007), no.~4, 771--786.

\bibitem{Hor} H\"ormander, L. {\it{The Analysis of Linear Partial
Differential Operators,}} Grundlehren der Mathematischen
Wissenschaften, Springer--Verlag, Berlin (1985).

\bibitem{Hor2} H\"ormander, L. {\it Oscillatory integrals and multipliers on $FL^p$},  Ark.\ Mat.~11, 1--11. (1973).

\bibitem{IS} Ionescu, A., Schlag, W. {\em Agmon--Kato--Kuroda theorems for a large class of
perturbations.},  Duke Math.\ J.~131  (2006),  no.~3, 397--440.

%\bibitem{Jen} Jensen, A. {\em Spectral properties of Schr\"odinger operators and time-decay of
%the wave functions results in $L^2(\R^m)$, $m\geq 5$.} Duke Math.\ J.~47  (1980), 57--80.

\bibitem{JenKat} Jensen, A., Kato, T. {\em Spectral properties of Schr\"odinger operators and time-decay of
the wave functions.} Duke Math.\ J.~46  (1979), no.~3, 583--611.

%\bibitem{JSS0}
%Journ\'e, J.-L., Soffer, A.; Sogge, C. D. {\em $L\sp p\to L\sp {p'}$
%estimates for time-dependent Schr\"odinger operators.}  Bull.\
%Amer.\ Math.\ Soc.\ (N.S.)  23  (1990),  no.~2, 519--524.

\bibitem{Kato} Kato, T. {\em Wave operators and similarity for some non-selfadjoint operators.} Math.\ Ann.~162
(1965/1966),  258--279.

\bibitem{KatoY} Kato, T., Yajima, K. {\em Some examples of smooth
operators and the associated smoothing effect.} Rev.\ Math.\ Phys.~1
(1989), no.~4, 481--496.

\bibitem{KochT} Koch, H., Tataru, D. {\em Carleman estimates and absence of embedded eigenvalues},
 Comm.\ Math.\ Phys.~267  (2006),  no.~2, 419--449.

%\bibitem{LossThall} Loss, M., Thaller, B. {\em Scattering of particles by long-range magnetic fields.}
%  Ann.\ Physics  176  (1987),  no.~1, 159--180.

\bibitem{RS4} Reed, M., Simon, B. {\em Methods of modern mathematical physics. IV. Analysis of operators.}
 Academic Press [Harcourt Brace Jovanovich, Publishers], New York-London, 1978.

\bibitem{Rob} Robert, D. {\em Asymptotique de la phase de diffusion \`a haute
\'energie pour des perturbations du second ordre du Laplacien}.
(French) Ann.\ Sci.\ \'Ec.\ Norm.\ Sup., IV S\'er.\ 25 (1992), No.
2, 107--134.

\bibitem{RS} Rodnianski, I., Schlag, W. {\em Time decay for solutions of
Schr\"odinger equations with rough and time-dependent potentials.}
Invent.\ Math.\ 155 (2004), no. 3, 451--513.

%\bibitem{RouYaf} Roux, Ph., Yafaev, D. {\em The scattering matrix for the Schr\"odinger operator
%with a long-range electromagnetic potential.}  J.\ Math.\ Phys.~44
%(2003),  no.~7, 2762--2786.

\bibitem{SJ} Sj\'olin, P. {\em Regularity of solutions to Schr\"odinger equations.} Duke Math.\ J.~55 (1987),
699--715.

\bibitem{Stef} Stefanov, A. {\em Strichartz estimates for the magnetic
Schr\"odinger equation}, preprint 2004.

\bibitem{strich} Strichartz, R. {\em Restrictions of Fourier transforms to quadratic
surfaces and decay of solutions of wave equations.}  Duke Math.\
J.~44  (1977), no.~3, 705--714.

\bibitem{Tay} Taylor, M.\ E. {\em Tools for PDE. Pseudodifferential operators, paradifferential operators, and layer potentials.}
 Mathematical Surveys and Monographs, 81. American Mathematical Society, Providence, RI, 2000.

\bibitem{VE} Vega, L. {\em Schr\"odinger equations: Pointwise
convergence to the initial data.} Proc.\ Amer.\ Math.\ Soc.~102
(1988), 874--878.

%\bibitem{Yaf} Yafaev, D.\ R. {\em Scattering matrix for magnetic potentials with
%Coulomb decay at infinity.}  Integral Equations Operator Theory  47
%(2003), no.~2, 217--249.

\end{thebibliography}
\end{document}